\documentclass{amsart}
\usepackage{amssymb,mathrsfs}
\usepackage{color}

\usepackage{enumitem}

\usepackage[textheight=25 cm, textwidth=16cm, headheight=10pt, headsep=20pt, footskip=10pt, left=1.5cm, right=1.5cm]{geometry}

\def\bN {\mathbf{N}}

\def\bQ {\mathbf{Q}}
\def\bR {\mathbf{R}}

\def\fH {\mathfrak{H}}

\def\\fH {\mathcal{H}}

\def\cP {\mathcal{P}}

\def\\hbar {{\\hbarilon}}

\def\si {{\sigma}}

\newcommand{\Lip}{\operatorname{Lip}}

\newcommand{\ba}{\begin{aligned}}
\newcommand{\ea}{\end{aligned}}

\newcommand{\be}{\begin{equation}}
\newcommand{\ee}{\end{equation}}

\newcommand{\bea}{\begin{eqnarray}}
\newcommand{\eea}{\end{eqnarray}}

\newcommand{\MKd}
{W_2}
\newcommand{\MKu}
{W_1}
\newcommand{\MKp}{W_p}

\newtheorem{Thm}{Theorem}[section]
\newtheorem{Rmk}[Thm]{Remark}
\newtheorem{Prop}[Thm]{Proposition}
\newtheorem{Cor}[Thm]{Corollary}
\newtheorem{Lem}[Thm]{Lemma}
\newtheorem{Def}[Thm]{Definition}

\renewcommand{\hbar}{{}}

\usepackage{hyperref}
\hypersetup{pdfborder=0 0 0}

\def\bN {\mathbf{N}}

\def\bQ {\mathbf{Q}}
\def\bR {\mathbf{R}}

\def\cP {\mathcal{P}}

\newcommand{\gen}{\gamma}
\newcommand{\vla}{\nu}
\newcommand{\vlau}{\mathbf v}
\newcommand{\ta}{\tau}
\newcommand{\hist}[2]{[#1]^{\leq #2}}
\newcommand{\histk}[2]{[#1]_k^{\leq #2}}

\newcommand{\vlav}[3]{\vlau(\hist{#3}{#1},#2)}
\newcommand{\vlavk}[3]{\vlau(\histk{#3}{#1},#2)}

\newcommand{\vlaun}{\mathbf \Gamma}
\newcommand{\histn}[2]{[(#1)^{\otimes N}]^{\leq #2}}
\newcommand{\histkn}[2]{[(#1)^{\times N}]_k^{\leq #2}}
\newcommand{\vlavn}[3]{\vlaun((\histn{#3}{#1})_{N;1},#2)}
\newcommand{\vlavkn}[3]{\vlaun^k((\histkn{#3}{#1})_{N:1},#2)}
\newcommand{\vlavkni}[3]{\vlaun^k_i((\histkn{#3}{#1})_{N:1},#2)}

\newcommand{\rhoi}{\rho^{in}}
\newcommand{\rhoin}{(\rhoi)^{\otimes N}}

\newcommand{\macphin}{\varphi^{in}}

\newcommand{\rc}{{\tau_c}}

\newcommand{\rod}{\tau}

\begin{document}

\title[The Mean-Field limit for  hybrid  models]{\Large The Mean-Field limit for  hybrid  models of collective motions with chemotaxis,}

\author[R. Natalini]{Roberto Natalini}
\address[R.N]{IAC, Via dei Taurini, 19, 00185 Roma RM Italy}
\email{roberto.natalini@cnr.it}
\author[T. Paul]{Thierry Paul}
\address[T.P.]{CNRS \& LJLL Sorbonne Université 4 place Jussieu 75005 Paris, France France}
\email{thierry.paul@upmc.fr}


\begin{abstract}
\Large
In this paper we study a general class of hybrid mathematical models of collective motions of cells under the influence of chemical stimuli. The models are hybrid in the sense that cells are discrete entities given by ODE, while the chemoattractant is considered as a continuous signal which solves a diffusive equation. For this model we prove the mean-field limit  in the Wasserstein distance to a system given by the coupling of a Vlasov-type equation with the chemoattractant equation. Our approach and results are  not based on empirical measures, but rather on marginals of large number of individuals densities, and we show the limit with explicit bounds, by proving also existence and uniqueness for the limit system. In the monokinetic case we derive new pressureless nonlocal Euler-type model with chemotaxis.

\end{abstract}
\LARGE
\maketitle

\tableofcontents


\section{Introduction}\label{intro}

A collective motion occurs when the behaviour of a group of individuals is dominated by the mutual interaction between them. This behaviour arises in many different contexts both for non-living and living systems, as for instance nematic fluids, simple robots, bacteria colonies, flocks of birds, schools of fishes, human crowds, see for instance \cite{vicsek-collective}. In a nutshell, all microscopic mathematical models of collective motion are based on one or more of the following elementary mechanisms: \emph{alignment}, see \cite{vicsek}, \cite{CS1}, and references therein, \emph{separation} and \emph{cohesion} \cite{dorsogna,strombom}. Concerning alignment models, a popular one is represented by the Cucker-Smale model \cite{CS1}, which was originally proposed to describe the dynamics in  flocks of birds, but then it was extended to cover more general phenomena, as for instance animal herding \cite{couzin-leader}. The hypothesis underlying to the Cucker-Smale model is that the force acting on every individual is a weighted average of the differences of its velocity with those of the others, and this force decays when the mutual distance between the individuals increases. Some preliminary analytical results about the time asymptotic behaviour of the model has been proven in  \cite{CS1,haliu}, and then a lot of papers investigated the behaviour of this dynamical model in many directions, see for instance \cite{carr3} and \cite{rt} for a comprehensive list of references.

In recent years, there was a lot of interest about collective motion of cells driven by chemical stimuli, see \cite{szabo,belmonte,arboleda,sepulveda,joie,colin,scianna}, and the reviews \cite{hatzikirou, vicsek-collective-cell}. Focusing on the family of Cucker-Smale models, in \cite{dicostanzo} a model for the morphogenesis in the zebrafish lateral line primordium was proposed, where a Cucker-Smale model was coupled with other cell mechanisms (chemotaxis, attraction-repulsion, damping effects) to describe the formation of neuromasts, see \cite{gilmour06, gilmour08} for the experimental basis of this model. The description of the cell behaviour is hybrid: while particles are considered discrete entities, endowed with a radius $R$ describing their circular shape, the chemical signal $\varphi$ is supposed to be continuous and its time derivative is equal to a diffusion term, a source term depending on the position of each particle, and a degradation term.  A simplified version of the model in  \cite{dicostanzo} was proposed in \cite{dicostanzo2} to allow a full analytical investigation. This simplified model reads as follows:

\begin{align}\label{sys-cuck-chemo}
\left\{
	\begin{array}{l}
	\dot{x}_i=v_i,\\
	\dot{v}_i=\displaystyle\frac{\beta}{N}\displaystyle\sum_{j=1}^N\frac{1}{\left(1+\frac{\left\|x_i-x_j\right\|^2}{R^2}\right)^{\sigma}}(v_j-v_i)+\eta \nabla_x \varphi (x_i),\\
	\partial_t \varphi=D\Delta \varphi - \kappa \varphi+ \displaystyle f(x,X(t) ),
	\end{array}
	\right.
\end{align}

Initial data are given by initial position and velocity for each particle:
\begin{align*}
	X(0)&=X_0,\quad 	 V(0)=V_0,
\end{align*}
with $X=\left(x_1,\dots,x_N\right)$, $V=\left(v_1,\dots,v_N\right)$, and by the initial concentration of signal, that it is assumed
\begin{align}\label{iniz-phi}
	\varphi(x,0):=\varphi_0=0.
\end{align}
Here $x_i,v_i$ are the position and velocity of the i-th cell and $\varphi$ stands for a generic chemical signal produced by the cells themselves and such that the cells are attracted towards the direction where $\nabla_x \varphi$ is growing. For this simple model in \cite{dicostanzo2} a full analytical theory was developed in the two-dimensional case with a fixed but arbitrary number $N$ of particles, and results of globally in time existence and uniqueness of solutions were proved, as well as the time-asymptotic linear stability. Other analytical results, for more general hybrid models, can be found in \cite{mp}.

 {In this paper we aim to prove the mean-field limit  for a general class of models including \eqref{sys-cuck-chemo} towards   Vlasov type kinetic  equations, together with the hydrodynamic mean-field limit of such models towards Euler type equations, coupled with chemotaxis.
To our knowledge, both limits, and the related   kinetic and Euler equations, and a fortiori their rigorous derivation,  are new in the literature.

Let us describe the class of particle systems we will handle in the present article.

Consider on $\bR^{2dN}\ni 
((x_i(t))_{i=1,\dots,N},(v_i(t))_{i=1,\dots,N}):=(X(t),V(t))$ the following vector field
\be\label{eq1}
\left\{
\begin{array}{l}
\dot {x_i}(t)=v_i\\
\dot {v_i}(t)=F_i(t,X(t),V(t))
\end{array}
 \ 
 i=1,\dots,N,\ 
 (X(0),V(0))=(X^{in},V^{in}):
\right.
\ee
}
where 
\be\label{defG}
F_i(t,X,V)=\frac1N\sum_{j=1}^N\gamma(v_i-v_j,x_i-x_j)+\eta
\nabla_{x}\varphi^t(x_i)
+F_{ext}(x_i),
\ee
$\gen$ is the collective interaction function, $F_{ext}$ is an external force  and $\varphi$ satisfies the equation
\be\label{defeqphi0}
\partial_s\varphi^s(x)=D\Delta_x\varphi-\kappa\varphi +f(x,X(s)),
\ s\in[0,t],\ \varphi^{s=0}=\macphin
\ee 
for some $\kappa, D,\eta\geq0$ and function $f$ of the form
\be\label{deff}
f(x,X)=\frac1N\sum_{j=1}^N\chi(x-x_i),\ \ \chi\in\mathcal C_c^1.
\ee

The  function $\gamma:\bR^d\times \bR^b\to\bR\times\bR^d$ is supposed to be Lipschitz continuous\footnote{
{through this paper we define $\Lip(f)$ for $f:\ \bR^{n}\to\bR^{m},m,n\in\bN,$ as $\Lip(f):=
\sqrt{\sum\limits_{i=1}^m(\Lip(f_i)^2}$.}}.

The case $\gamma(y,w)=\psi(y)w.\ F=\varphi=0$, $\psi$ bounded Lipschitz, covers the standard case of Cucker-Smale models.

\vskip 1cm
For any fixed function $\macphin$ and any  $t,N$ we  define the mapping $\Phi^t_N=\Phi^t$  by
\be\label{defphitN}
\left\{
\begin{array}{rrcl}
\Phi^t_N:&\bR^{2dN}&\longrightarrow&\bR^{2dN}\\
&Z^{in}=(X^{in},V^{in})&\longrightarrow& Z(t)=(X(t),V(t))\mbox{ solution of $(\ref{eq1},\ref{defG},\ref{defeqphi0},\ref{deff})$.}
\end{array}
\right.
\ee
Note that $\Phi^t_N$ is not a  flow.
\vskip 1cm
We would like to derive a kinetic model corresponding to  system $(\ref{eq1},\ref{defG},\ref{defeqphi0},\ref{deff})$, that is the one particle (non-linear) PDE satisfied by the first marginal of the push-froward\footnote{We recall that the pushforward of a measure $\mu$ by a measurable function $\Phi$ is $\Phi\#\mu$ defined by $\int \varphi d(\Phi\#\mu):=\int (\varphi\circ \Phi)d\mu$ for every test function $\varphi$.} $\Phi^t\#\rho^{in}$ where 
$\rho^{in}\in\mathcal P(\bR^{2dN})$, the space of probability measures on $\bR^{2dN}$ and $\Phi^t_N$ is the mapping defined by \eqref{defphitN}.
\vskip 0.3cm
The first difficulty is the fact that $\rho^t_N:=\Phi^t_N\#\rho^{in}_N$ does not satisfy a closed PDE, except if $\rho^{in}_N=\rho_{\bar Z}$ where
\be\label{ansatz}
\rho_{\bar Z}:=\frac1{N!}\sum_{\Sigma\in\Sigma_N}\delta_{\sigma(\bar Z)}, \ \bar Z:=(\bar X,\bar V)\in\bR^{2dN}.
\ee 
Here $\Sigma_N$ is the group of permutations  of $N$ elements and 
\be\label{defperm}
\sigma(\bar Z)=\sigma(\bar X,\bar V):=(\bar x_{\sigma(1)},\dots,\bar x_{\sigma(N)},\bar v_{\sigma(1)},\dots,\bar v_{\sigma(N)}).
\ee

In this case $\rho^t_N:=\rho_{\Phi^t(\bar Z)}$  satisfies 
\be\label{liouvansatz}
\partial_t\rho^t_N+V\cdot\nabla_X\rho^t_N=
 \sum_{i=1}^N\nabla_{v_i}\cdot G_i\rho^t_N
\ee
where 
\be\label{defG2}
G_i(t,X,V)=\frac1N\sum_{j=1}^N\gamma(v_i-v_j,x_i-x_j)+\eta\nabla_{x}\Psi^t(x_i)
+F_{ext}(x_i),
\ee
 $\Psi$ (and therefore $G_i$ too)  depends on the solution $\rho^t_N$ and satisfies the equation
\be\label{defeqphi}
\partial_t\Psi^t(x)=D\Delta_x\Psi-\kappa\Psi +f(x,\rho^s_{N;1}),\ s\in[0,t],
\ee 
with $g$ given by
\be\label{fansatz}
f(x,\rho^s_{N;1})=
\int_{\bR^{2d}}\chi(x-y)\rho^s_{N;1}(y,\xi)dyd\xi,
\ \
\ee
where, denoting $\Phi^t_N(\bar Z)=(\bar x_1(t),\dots,\bar x_N(t),\bar v_1(t),\dots,\bar v _N(t))$,
\begin{eqnarray}
&&\rho^s_{N;1}(y,\xi)\nonumber\\
&:=&\int_{\bR^{2d(N-1)}}
\rho^t_N(y,x_2,\dots,x_N,\xi,v_2,\dots,v_N)dx_2\dots dx_N dv_2\dots dv_N\nonumber\\
&=&\int_{\bR^{2d(N-1)}}
\rho_{\Phi^t_N(\bar Z)}(y,x_2,\dots,x_N,\xi,v_2,\dots,v_N)dx_2\dots dx_N dv_2\dots dv_N \ 
\mbox{(see Lemma \ref{deltempir} below)}
\nonumber\\
&=&
\tfrac1N\sum_{i=1}^N\delta(y-\bar x_i(t))\delta(\xi-\bar v_i(t))=:\mu_{\Phi^t_N(\bar Z)}
\ \ \ \ \ \ \mbox{ (see Lemma \ref{deltempir} below)}.
\nonumber
\end{eqnarray}
In turn, this suggests that the (non local in time) Vlasov equation associated to the particle system $(\ref{eq1},\ref{defG},\ref{defeqphi0},\ref{deff})$ is 
\be\label{eq1V}
\partial_t\rho^t+v\cdot\nabla_x\rho^t=\nabla_v(
\vla(t,x,v)\rho^t),\ \rho^{0}=\rhoi
\ee
where
\be\label{defGV}
\vla(t,x,v)=\gamma*\rho^t(x,v)+\eta\nabla_x\psi^t(x)+F_{ext}(x)
\ee
and  $\psi^s$ satisfies
\be\label{defeqphiV}
\partial_s\psi^s(x)=D\Delta_x\psi-\kappa\psi^s+g(x,\rho^s),\ \psi^0=\macphin
.
\ee
with
\be\label{defgvlasov}
g(x,\rho^s)=
\int_{\bR^{2d}}\chi(x-y)\rho^s(y,\xi)dyd\xi.
\ee

%

{The kinetic equation associated to Cucker-Smale systems, introduced in \cite{ht}, has been derived in \cite{haliu,hakimzhang} and, for generalizations  of type  \eqref{eq1} with $\varphi=0$ in \cite{rt}, but, in all these papers, {\bf without chemotaxis interaction}. We refer to \cite{hakimzhang,rt} for a large bibliography on the subject.}

Let us finish this section by recalling  the three following dynamics involved in this paper, denoted by (P) for \textit{(Particles)}, (LV) for \textit{(Liouville-Vlasov)} and (V) for \textit{(Vlasov)} and the strategy adopted in the proof of the main results: 
$$
\begin{array}{cl}
(P)& \left\{ 
\begin{array}{l}
\dot {x_i}=v_i,\  
\dot {v_i}=F_i(t,X(t),V(t)),\ 
(X(0),V(0))=Z(0)=Z^{in}\in\bR^{2dN} 
\\ \\
F_i(t,X,V)=\frac1N\sum\limits_{j=1}^N\gamma(v_i-v_j,x_i-x_j)+\eta\nabla_{x}
\varphi^t(x)_i
+F_{ext}(x_i),\\ \\
\partial_s\varphi^s(x)
=D\Delta_x\varphi-\kappa\varphi +f(x,X(s)),\ s\in[0,t],,\ \varphi^0=\macphin, \\\ \\
f(x,X)=\frac1N\sum\limits_{j=1}^N\chi (x-x_j);
\end{array} 
\right.
\nonumber
\\ \\ \\
(LV)& \left\{
\begin{array}{l}
\partial_t\rho^t_N+V\cdot\nabla_X\rho^t_N=
 \sum\limits_{i=1}^N\nabla_{v_i}\cdot G_i\rho^t_N,\ \rho^o_N=\rho_N^{in}=(\rho^{in})^{\otimes N}\in\cP(\bR^{2dN})\\ \\
G_i(t,X,V)=\frac1N\sum\limits_{j=1}^N\gamma(v_i-v_j,x_i-x_j)+\eta\nabla_{z}
\Psi^t(x_i)
+F_{ext}(x_i),
\\ \\
\partial_s\Psi^s(x)=D\Delta_x\Psi-\kappa\Psi +g(x,\rho^s_{N;1}),\ s\in[0,t],\ \Psi^0=\macphin,
\\ \\
g(x,\rho^s_{N;1})=
\chi*\rho^s_{N;1}(x);
\end{array}
\right.
\nonumber\\ \\ \\
(V)& \left\{
\begin{array}{l}
\partial_t\rho^t+v\cdot\nabla_x\rho^t=\nabla_v(
\vla(t,x,v)\rho^t),\ \rho^0=\rho^{in}\in\cP(\bR^{2d})\\ \\
\vla(t,x,v)=\gamma*\rho^t(x,v)+\eta\nabla_x\psi^t(x)+F_{ext}(x),\\ \\
\partial_s\psi^s(x)=D\Delta_x\psi-\kappa\psi+g(x,\rho^s),\ \psi^0=\macphin.\\ \\ 
g(x,\rho^s)=
\chi*\rho^s(x).
\end{array}
\right.\nonumber
\end{array}
$$
Note that
$
(\chi*\rho_{N;1}(t))(x)
=
(\widetilde \chi*\rho)(x,\dots,x),\ \ \ \widetilde \chi(X)=\frac1N\sum\limits_{j=1}^N\chi(x_j).
$
\vskip 1cm
The strategy of our approach can be summarized by the following estimates that we will establish  in some Wasserstein topology, 
\begin{eqnarray}
\left\{
\begin{array}{rcll}
(\Phi_N^t\#\rho^{in}_N)_{N;1}&\sim& (\rho^t_N)_{N;1}(t),\ \ &\Phi_N^t\mbox{ solution of }(L),\ \rho_{N}^t\mbox{  of }(LV)\mbox{ with }\rho_N^0=\rho^{in}_N\\ 
\rho^{in}_N&=&(\rho^{in})^{\otimes N}\\ 
(\rho^t_N)_{N;1}&\sim&\rho^t,\ \ &\rho^t\mbox{ solution of }(V)\mbox{ with }\rho^0=\rho^{in},
\end{array}
\right.\nonumber
\end{eqnarray}
so that, by triangle inequality,
$$
(\Phi_N^t\#\rho^{in}_N)_{N;1}\sim \rho^t
$$
$  \mbox{ with,}  
\Phi_N^t\mbox{ solution of }(L)\mbox{ and } \rho^t\mbox{ solution of }(V)\mbox{ with }\rho^0=(\rho^{in}_N)_{N;1}.$

\section{The main results}

\begin{Thm}\label{mainthm}

Let $\rho^{in}$ be a compactly supported 
probability on $\bR^{2dN}$,  let $\Phi^t_N$ be the mapping generated by the particles system $(\ref{eq1},\ref{defG},\ref{defeqphi0},\ref{deff})$ as defined by \eqref{defphitN},
 and let 
 $\tau_{\rho^{in}}$ be the function defined in formula \eqref{PtoVtau} below.
 
Then, for any $t\geq 0$,
$$
\MKd\left((\Phi^t_N\#(\rho^{in})^{\otimes N})_{N;1},\rho^t\right)^2\leq 
\tau_{\rho^{in}}(t)
\left\{
\begin{array}{ll}
N^{-\frac12}&d=1\\
N^{-\frac12}{\log{N}}&d=2\\
N^{-\frac1{d}}&d>2
\end{array}
\right.
$$
where $\rho^t$ is the solution of the   Vlasov equation 
$(\ref{eq1V},\ref{defGV},
\ref{defeqphiV},\ref{defgvlasov})$ with initial condition $\rho^{in}$  provided by Theorem \ref{thmvlasov} below and $\MKd$ is the quadratic Wasserstein distance whose definition is recalled in Definition \ref{defwas}.

Moreover, let us 
denote by $\varphi^t_{Z^{in}}$ the chemical density solution of $(\ref{eq1},\ref{defG},\ref{defeqphi0},\ref{deff})$ with initial data $(Z^{in},\macphin)$ and by $\psi^t_{\rhoi}$ the one solution of $(\ref{eq1V},\ref{defGV},
\ref{defeqphiV},\ref{defgvlasov})$ with initial data $(\rhoi,\macphin)$.

Then
$$
\int_{\bR^{2dN}}\|\nabla\varphi^t_{Z^{in}}-\nabla\psi^t_{\rhoi}\|^2_\infty
(\rhoi)^{\otimes N}(dZ^{in})
\leq 
\rc(t)
\left\{
\begin{array}{ll}
N^{-\frac12}&d=1\\
N^{-\frac12}{\log{N}}&d=2\\
N^{-\frac1{d}}&d>2
\end{array}
\right.
$$
where $\rc$
is defined below by \eqref{tauct}.

Finally, the functions $\tau(t),\rc(t)$ depend only on $t$, $\Lip(\gen),\Lip(\chi),\Lip(\nabla\chi)$, and the supports of $\Phi^t_N\#\rhoin$ and $\rho^t$, and satisfies the following estimate for all $t\in\bR$,
$$
\tau_{\rhoi}(t)\leq e^{e^{Ct}},\ \rc(t)\leq e^{e^{C_ct}}
$$
for some constants $C,C_c$,depending  on  $\Lip(\gen),\Lip(\chi),\Lip(\nabla \chi)$ and $|supp(\rhoi)|$.
\end{Thm}
\vskip 1cm
\begin{Cor}[Hydrodynamic Euler limit]\label{cormain}\ 
{
Let $\mu^{in},u^{in},\varphi^{in}$ be such that the non local  Euler system
\be\label{systhydro}
\left\{
\begin{array}{l}
\partial_t\mu^t+\nabla(u^t \mu^t) = 0\\
\partial_t(\mu^tu^t)+
\nabla(\mu^t(u^t)^{\otimes 2}) = 
\mu^t\int\gamma(\cdot-y,u^t(\cdot)-u^t(y))\mu^t(y)dy+
\eta\mu^t\nabla\psi^t
+\mu^tF\\
\partial_s\psi^s = D\Delta\psi-\kappa\psi +
\chi*\mu^s,\ s\in[0,t],
\\ 
(\mu^0,u^0,\psi^0) = (\mu^{in},u^{in},\varphi^{in})\in H^s,\ s>\tfrac d2+1.
\end{array}
\right.\nonumber
\ee
 has a unique solution $
\mu^t, u^t\in C([0,t];H^s)\cap C^1([0,T];H^{s-1}),\ 
\psi^t  \in C([0,t];H^{s})\cap C^1([0,T];H^{s-2}) \cap L^2(0,T; H^{s+1})$
%
and let 
}
$$
\rhoi=\mu^{in}(x)\delta(v-u^{in}(x)).
$$
Then, for any $t\in [ 0,T]$,
$$
\MKd\left((\Phi^t_N\#(\rho^{in})^{\otimes N})_{N;1},\mu^t(x)\delta(v-u^t(x))\right)^2\leq 
\tau(t)
\left\{
\begin{array}{ll}
N^{-\frac12}&d=1\\
N^{-\frac12}{\log{N}}&d=2\\
N^{-\frac1{d}}&d>2
\end{array}
\right.
$$
Moreover, 
$$
\int
_{\bR^{dN}}
\|\nabla\varphi^t_{(X^{in},u^{\otimes N}(X^{in}))}-\nabla\psi^t_{\rhoi}\|^2_\infty
(\mu^{in})^{\otimes N}(dX^{in})
\leq 
\rc(t)
\left\{
\begin{array}{ll}
N^{-\frac12}&d=1\\
N^{-\frac12}{\log{N}}&d=2\\
N^{-\frac1{d}}&d>2
\end{array}
.
\right.
$$
\end{Cor}
\vskip 1cm
\begin{proof}
[Proof of Theorem \ref{mainthm}]

Clearly Theorem \ref{mainthm} links the dynamics of the particle system  $(\ref{eq1},\ref{defG},\ref{defeqphi0},\ref{deff})$ and the one driven by the Vlasov system $(\ref{eq1V},\ref{defGV},
\ref{defeqphiV},\ref{defgvlasov})$.
As an intermediate step 
we will consider the  $N$-body Liouville type one defined by $(\ref{liouvansatz},\ref{defG2},\ref{defeqphi},\ref{fansatz})$.

We will  proceed in several steps.
\vskip 0.3cm
Step $1$ \textit{[Section \ref{secptolv}]}: we will show that
the marginal $(\Phi^t_N\#(\rho^{in})^{\otimes N})_{N;1}$  of the pushforward of the initial condition by the flow generated by the particle system $(\ref{eq1},\ref{defG},\ref{defeqphi0},\ref{deff})$ and the marginal $(\rho^t_{N;1})$ of the solution $\rho^t_N$ of $(\ref{liouvansatz},\ref{defG2},\ref{defeqphi},\ref{fansatz})$ are close as $N\to\infty$ in the same Wasserstein topology
through an estimate for
$
\MKd((\Phi^t_N\#(\rho^{in})^{\otimes N})_{N;1},(\rho^t_N)_{N;1})
$.
\vskip 0.3cm
Step $2$ \textit{[Section \ref{seclvtov}]} we will show that
the marginal $(\rho^t_{N;1})$ of the solution $\rho^t_N$ of $(\ref{liouvansatz},\ref{defG2},\ref{defeqphi},\ref{fansatz})$, is close to the solution of  a Vlasov type closed equation $(\ref{eq1V},\ref{defGV},
\ref{defeqphiV},\ref{defgvlasov})$ derived below in Wasserstein metric
by estimating
$
\MKd((\rho^t_N)_{N;1},\rho^t).
$
\vskip 0.3cm
Step $3$:
\textit{[particle density]}:
 we will use the triangular inequality for $\MKd$:
$$
\MKd((\Phi^t_N\#(\rho^{in})^{\otimes N})_{N;1},\rho^t)\leq 
\MKd((\Phi^t_N\#(\rho^{in})^{\otimes N})_{N;1},(\rho^t_N)_{N;1})
+
\MKd((\rho^t_N)_{N;1},\rho^t).
$$
The first part of Theorem \ref{mainthm} is then  given by the estimate given by Proposition \ref{ptolv} - namely  $\MKd((\Phi^t_N\#(\rho^{in})^{\otimes N})_{N;1},(\rho^t_N)_{N;1})\leq \beta_{\rho^{in}}(t)\sqrt{C_d(N)}$ -  and the one given by Proposition \ref{lvtov} - namely  $\MKd((\rho^t_N)_{N;1},\rho^t)\leq \frac{\alpha_{\rho^{in}}(t)}{\sqrt N}$ . 
\vskip 0.3cm
Step $4$:
\textit{[chemical density]}:
the chemical density estimate is obtained through the triangle inequality. We get  
\begin{eqnarray}
\|\nabla\varphi^t_{Z^{in}}-\nabla\psi^t_{\rhoi}\|_\infty^2
&\leq& \left(\|\nabla\varphi^t_{Z^{in}}-\nabla\psi^t_{\mu^{Z^{in}}}\|_\infty+\|\nabla\psi^t_{\mu_{Z^{in}}}-\nabla\psi^t_{\rhoi}\|_\infty\right)^2\nonumber\\
&\leq&
2
\left(\|\nabla\varphi^t_{Z^{in}}-\nabla\psi^t_{\mu^{Z^{in}}}\|_\infty^2+\|\nabla\psi^t_{\mu_{Z^{in}}}-\nabla\psi^t_{\rhoi}\|_\infty^2\right),\label{eq3}
\end{eqnarray}
where $\mu_{Z^{in}}:=\frac1N\sum\limits_{l=1}^N\delta_{z^{in}_l}$ and $\psi^t_{\rhoi}$ solves  $(\ref{liouvansatz},\ref{defG2},\ref{defeqphi},\ref{fansatz})$ with initial condition $\rhoin=\mu_{Z^{in}}^{\otimes N}$.

 Both terms  in the right hand-side of \eqref{eq3} are estimated by  Corollary \ref{corder}:
 \begin{eqnarray}
 \|\nabla\varphi^t_{Z^{in}}-\nabla\psi^t_{\mu^{Z^{in}}}\|_\infty^2
 &\leq& t^2\Lip(\nabla\chi)^2\MKd((\Phi^t_N\#(\mu_{Z^{in}})^{\otimes N})_{N:1},\rho^t_{\mu_{Z^{in}}})^2\label{eq11}\\
 \|\nabla\psi^t_{\mu_{Z^{in}}}-\nabla\psi^t_{\rhoi}\|_\infty^2&\leq&
 t^2\Lip(\nabla\chi)^2\MKd(\rho^t_{\mu_{Z^{in}}},\rho^t_{\rho^{in}})^2, \label{eq2}
 \end{eqnarray}
 where $\rho^t_{\mu_{Z^{in}}}$ (resp. $\rho^t_{\rho^{in}}$) is the solution of the Vlasov equation with initial condition ${\mu_{Z^{in}}}$ (resp. $\rho^{in}$).

   $\MKd((\Phi^t_N\# (\mu_{Z^{in}})^{\otimes N},\rho^t_{\mu_{Z^{in}}})^2$ is estimated by the first estimate of  Theorem \ref{mainthm} we just proved in Step $3$ - namely 
   $\MKd((\Phi^t_N\# \mu_{Z^{in}},\rho^t_{\mu_{Z^{in}}})^2\leq 
\tau_{\mu_{Z^{in}},}(t)
{\scriptsize \left\{
\begin{array}{ll}
N^{-\frac12}&d=1\\
N^{-\frac12}{\log{N}}&d=2\\
N^{-\frac1{d}}&d>2
\end{array}
\right.}$ -, while $\MKd(\rho^t_{\mu_{Z^{in}}},\rho^t_{\rho^{in}})^2$ by the Dobrushin estimate in Theorem \ref{thmvlasov} -namely 
$.\MKd(\rho^t_{\mu_{Z^{in}}},\rho^t_{\rho^{in}})^2\leq 2e^{\Gamma(t)}\MKd(\mu_{Z^{in}},\rho^{in})^2$.

Therefore, by \eqref{eq3},\eqref{eq2} and \eqref{eq11},
\begin{eqnarray}
\|\nabla\varphi^t_{Z^{in}}-\nabla\psi^t_{\rhoi}\|_\infty^2
&\leq&
2t^2\Lip(\nabla\chi)^2
\left(
2e^{\Gamma(t)}\MKd(\mu_{Z^{in}},\rho^{in})^2
+
\tau_{\mu_{Z^{in}}}(t)
{\scriptsize \left\{
\begin{array}{ll}
N^{-\frac12}&d=1\\
N^{-\frac12}{\log{N}}&d=2\\
N^{-\frac1{d}}&d>2
\end{array}
\right.}
\right)\nonumber
\end{eqnarray}
and
\begin{eqnarray}
&&\int_{\bR^{2dN}}\|\nabla\varphi^t_{Z^{in}}-\nabla\psi^t_{\rhoi}\|^2_\infty
(\rhoi)^{\otimes N}(dZ^{in})
\nonumber\\
&\leq &
2t^2\Lip(\nabla\chi)^2
\left(
2e^{\Gamma(t)}CM_2(\rhoi){\scriptsize \left\{
\begin{array}{ll}
N^{-\frac12}&d=1\\
N^{-\frac12}{\log{N}}&d=2\\
N^{-\frac1{d}}&d>2
\end{array}
\right.}
+
\tau_{\mu_{Z^{in}}}(t)
{\scriptsize \left\{
\begin{array}{ll}
N^{-\frac12}&d=1\\
N^{-\frac12}{\log{N}}&d=2\\
N^{-\frac1{d}}&d>2
\end{array}
\right.}
\right)\nonumber\\
&=&
2t^2\Lip(\nabla\chi)^2(2e^{\Gamma(t)}CM_2(\rhoi)+\tau_{\mu_{Z^{in}}}(t)){\scriptsize \left\{
\begin{array}{ll}
N^{-\frac12}&d=1\\
N^{-\frac12}{\log{N}}&d=2\\
N^{-\frac1{d}}&d>2
\end{array}
\right.}\nonumber\\
&\leq&
2t^2\Lip(\nabla\chi)^2(2e^{\Gamma(t)}CM_2(\rhoi)+\bar\tau_{\rhoi}(t)){\scriptsize \left\{
\begin{array}{ll}
N^{-\frac12}&d=1\\
N^{-\frac12}{\log{N}}&d=2\\
N^{-\frac1{d}}&d>2
\end{array}
\right.}=:\rc(t){\scriptsize \left\{
\begin{array}{ll}
N^{-\frac12}&d=1\\
N^{-\frac12}{\log{N}}&d=2\\
N^{-\frac1{d}}&d>2
\end{array}
\right.}\label{tauct}
\end{eqnarray}
by \eqref{bbar} since one integrates in $Z^{in}$ on the support of $(\rhoi)^{\otimes N}$ so that $supp(\mu_{Z^{in}})\subset supp(\rhoi)$.
\vskip 0.3cm
Step $5$:
\textit{[rate of convergence]}:
the estimate for $\tau_{\rhoi}(t)$ is proven at the end of Section \ref{seclvtov} (see formula \eqref{estau}), the one for $\rc(t)$ follows by \eqref{tauct}. \end{proof}

\begin{proof}[Proof of 
 Corollary \ref{cormain}]
 Corollary \ref{cormain} is a  rephrasing of Theorem \ref{mainthm} in the monokinetic case, which is straightforward by using  Theorem \ref{hydro}.\end{proof}

\begin{Rmk}\label{alter}
As it is clear from the step 3 above, an alternative to the second statement in Theorem \ref{mainthm} is the following.
$$
\|\nabla\varphi^t_{Z^{in}}-\nabla\psi^t_{\rhoi}\|^2_\infty
\leq 
2t^2\Lip(\nabla\chi)^2\left(2e^{\Gamma(t)}\MKd(\mu_{Z^{in}},\rho^{in})^2+
\tau_{\rhoi}(t)\left\{
\begin{array}{ll}
N^{-\frac12}&d=1\\
N^{-\frac12}{\log{N}}&d=2\\
N^{-\frac1{d}}&d>2
\end{array}
\right.\right)
$$
for each $Z^{in}\in\ supp((\rhoi)^{\otimes N})$.
\end{Rmk}

\section{Technical Preliminaries}\label{prelim}
In this section we establish or recall  several results which will be intensively used in the core of the proof of Theorem \ref{mainthm}.
\subsection{Wasserstein distances}
Let us start this section by recalling the definition of the first and second order Wasserstein distance $\MKd$ (see \cite{VillaniAMS,VillaniTOT}).
\begin{Def}[quadratic Wasserstein distance]\label{defwas}
The Wasserstein distance of order two between two probability measures $\mu,\nu$ on $\bR^m$ with finite second moments is defined as
$$
\MKd(\mu,\nu)^2
=
\inf_{\gamma\in\Gamma(\mu,\nu)}\int_{\bR^m\times\bR^m}
|x-y|^2\gamma(dx,dy)
$$
where $\Gamma(\mu,\nu)$ is the set of probability measures on $\bR^m\times\bR^m$ whose marginals on the two factors are $\mu$ and $\nu$.
\end{Def}
Likewise  is the first order Wasserstein distance  $\MKu$  between two probability measures $\mu,\nu$ on $\bR^m$ with finite  moments is defined by the following.

\begin{Def}\label{defwasun}
$$
\MKu(\mu,\nu):=\sup\{\int_{\bR^{2d}}f(\mu-\nu)|\ f\in C^\infty(\bR^{2d}),\ \Lip(f)\leq 1\}.
$$
\end{Def}
\begin{Lem}\label{lemfund}\

\begin{enumerate}
[label=(\roman*)]

\item\hskip 1cm{$
\MKu(\mu,\nu)\leq\MKd(\mu,\nu),
$}
\item\label{33ii}
\hskip 1cm{$
\sup\limits_{\Lip{f}\leq 1}|\int f(\mu-\nu)|=\MKu(\mu,\nu)\leq\MKd(\mu,\nu),
$}
\item
\hskip 1cm The convergence in the weak topology  (i.e., in
the duality with $C_b(R^{2d})$ of sequences of probability measures with supports equibounded is equivalent to the convergence with respect to the distance $\MKp,\ p=1,2$ (in fact with respect to  Wasserstein of all orders),
\end{enumerate}
\end{Lem}
\begin{proof}
The first and second items are exactly formulas (7.1) and (7.3) in \cite{VillaniAMS}, The third item is a straightforward consequence of \cite[Theorem 7.12 (iii)]{VillaniAMS}, since the weak convergence of equisupported sequences of measures implies the convergence of all of their moments
\end{proof}
Note that Lemma \ref{lemfund} implies that
\be\label{rabin}
|\int f(\mu-\nu)|\leq \Lip{f}\MKu(\mu,\nu)\leq\Lip{f}\MKd(\mu,\nu)
\ee
\label{prelimwass}
\subsection{The diffusion term}\label{prelimdiff}
The three equations \eqref{defeqphi0}, \eqref{defeqphi}, \eqref{defeqphiV}, namely
\be\label{threeq}
\left\{
\begin{array}{rcll}
\partial_s\varphi^s(z)
&=&D\Delta_z\varphi-\kappa\varphi +f(z,Y(s)),&\ \varphi^0=\macphin\\
\partial_s\Psi^s(z)&=&D\Delta_z\Psi-\kappa\Psi +g(z,\rho^s_{N;1}),&\ \Psi^0=\macphin\\
\partial_s\psi^t(z)&=&D\Delta_z\psi-\kappa\psi+g(z,\rho^s),&\ \psi^0=\macphin
\end{array}
\right.
\ee

can be solved, denoting ${\bf \mathbb\bf  I}=\begin{pmatrix}1\\1\\1\end{pmatrix}
$, by
\begin{eqnarray}\label{solvedas}
\begin{pmatrix}
\varphi^t(z)
\\
\Psi^t(z)\\
\psi^t(z)
\end{pmatrix}
&=&
e^{-\kappa t}\int_0^te^{(t-s)D\Delta_z}
\begin{pmatrix}
f(z,X(s))\\
g(z,\rho^s_{N;1})\\
g(z,\rho^s)
\end{pmatrix}ds
+
e^{-\kappa t}e^{tD\Delta}\macphin {\bf \mathbb\bf  I}
\\
&=&
e^{-\kappa t}\int_0^t\int_{\bR^d}\tfrac{e^{-\frac{(z-z')^2}{4D(t-s)}}}{(4\pi D(t-s))^{\frac d2}}
\begin{pmatrix}
f(z',X(s))\\
g(z',\rho^s_{N;1})\\
g(z',\rho^s)
\end{pmatrix}dsdz'
+
e^{-\kappa t}e^{tD\Delta}\macphin.{\bf \mathbb\bf  I}
\nonumber
\end{eqnarray}


Note that $\nabla_z\begin{pmatrix}
\varphi^t(z)
\\
\Psi^t(z)\\
\psi^t(z)
\end{pmatrix}$ is given by the same formula after replacing $\chi$ by $\nabla\chi$ in the definitions of $f$ and $g$. 
%
%
\vskip 1cm
The following lemma will be systematically used inthe forthcoming sections.

\begin{Lem}\label{systemat}
Let $\rho,\rho'\in\cP(\bR^d)$ and $\mu\in\Lip(\bR^d)$. Then,  for all $t\geq 0$,
$$
\|(e^{t\Delta}\mu)*(\rho-\rho')\|_{L^\infty(\bR^d)}
\leq
\Lip(\mu)\MKp(\rho,\rho'),\ p=1,2.
$$
\end{Lem}
\begin{proof}
On has
\begin{eqnarray}
|(e^{t \Delta}\mu) *(\rho-\rho')(x_i)|
&=&
|\int(e^{t\Delta}\mu)(x_i-z)(\rho-\rho')dz|\nonumber\\
&\leq&
\Lip{((e^{t\Delta}\mu)(x_i-\cdot))}\MKd(\rho,\rho')\nonumber\\
&\leq&
\Lip{(e^{t\Delta}\mu)}\MKd(\rho,\rho')\nonumber\\
&\leq&
\Lip{\mu}
\MKd(\rho,\rho')\nonumber
\end{eqnarray}
since, by  
Lemma \ref{lemfund},
$$
\sup_{\Lip{f}\leq 1}\left|\int f(d\mu-d\nu)\right|=\MKu(\mu,\nu)\leq\MKd(\mu,\nu),
$$
and 
\begin{eqnarray}
|(e^{t\Delta}\mu)(x)-(e^{t\Delta}\mu)(y)|
&=&
|(e^{t\Delta}(\mu(x-\cdot)-e^{t\Delta}(\mu(y-\cdot))(0)|\nonumber\\
&\leq& |\mu(x)-\mu(y)|\leq\Lip(\mu)|x-y|.\nonumber
\end{eqnarray}

\end{proof}
\begin{Cor}\label{corder}Let $\varphi^t$ and $\si^t$ solve \eqref{threeq}. Then
\be\nonumber
\|\nabla\varphi^t-\nabla\psi^t\|_{L^\infty(\bR^d)}
\leq
t\Lip(\nabla\chi)\MKd((\Phi^t_N\#(\rhoi)^{\otimes N})_{N:1},\rho^t).
\ee
\end{Cor}
%
\subsection{Propagation of Wasserstein type estimates}
\label{prelimpropa}
In this paragraph, we establish a result used later as a black box, concerning the propagation of estimates in Wasserstein topology under general  transport 
equation including the several types used in this paper.
\begin{Thm}\label{thmpropwas}

Let us define the set of compactly supported probability measure on $\bR^{2dN}$ invariant by permutation:
$$
\cP^p_c(\bR^{2dN}):=\{\rho\in\cP_c(\bR^{2dN}),\rho(\sigma(dZ))=\rho(dZ),\ \forall \sigma\in\Sigma_N\}
$$
where $\sigma(Z)$ is defined in \eqref{defperm}.

 Let us suppose that the two equations
\be\label{twoeq}
\begin{array}{l}
\partial_t\rho^t_i+V\cdot\nabla_X\rho^t_i=\nabla_V.\cdot(\vlau_i([\rhoi_i]^{\leq t},X,V)\rho^t_i),\ \rho^0_i=
\rhoi_i\in\cP_c^p(\bR^{2dN}),\ i=1,2,
\end{array}
\ee
have the property of existence and uniqueness of solutions in $C^0(\bR^+,\cP_c^p(\bR^{2dN}))$.

Here 
\be\label{defcroch}
[\rhoi_i]^{\leq t}: \ s\in[0,t]\to \rho_i^s,\ i=1,2,
\ee 
and 
$\vlau_i([\rhoi_i]^{\leq t},X,V)$ 
is supposed to be invariant by permutations of the variables $(x_j,v_j)$, $ j=1,\dots, N$,   is  Lipschitz continuous with respect to $(X,V)$ and satisfies the estimate
$$
\vlau_i(\psi^{\leq t},X,V)\leq \gen_0\|V\|,
\ i=1,2
$$
for some constant $\gen_0<\infty$, uniformly in $X,V\in\bR^{2dN},\psi^{\leq t}:\ [0,t]\to\cP^p_c(\bR^{2dN}),t\in\bR$.


Let us finally define, for $i=1,2$,
$$
(\rho^t_i)_{N:1}(x,v):=\left\{\begin{array}{ll}
\int_{\bR^{2d(N-1)}}\ 
\rho^t(x,x_2,\dots,x_n;v,v_2,\dots,v_N)
dx_2\dots dx_N dv_2\dots dv_N&N>1\\
\rho_i^t(x,v),& N=1.
\end{array}
\right.
$$
Then, for all $t\in\bR^+$, and all $i=1,2$,

\begin{eqnarray}
&&\MKd((\rho^t_1)_{N:1},(\rho^t_2)_{N:1})^2\leq e^{\int_0^tL_1(s)ds}
\MKd(\rhoi_1,\rhoi_2)^2\nonumber\\
&&+
\frac2N\int_0^t\int_{\bR^{2dN}}|\vlau_1([\rhoi_1]^{\leq s},Y,\Xi)-\vlau_2([\rhoi_2]^{\leq s},Y,\Xi)| ^2
\rho_2^s
(dY,d\Xi)e^{\int_s^tL_1(u)du}ds\nonumber
\end{eqnarray}
with
$$
L_1(t)=
2(1+ \sup_{(X,V)\in\ supp(\rho_1^t)}
(\Lip(\vlau_1
(t,X,V)
)_{(X,V)})^2).
$$
\end{Thm}
The proof on Theorem \ref{thmpropwas} is given in Appendix \ref{proofthmpropwas}.
\section{From particles to Liouville-Vlasov
}\label{secptolv}
In this section we estimate $\MKd((\Phi_N^t\#\rho^{in})_{N;1},(\rho^t_N)_{N:1})$, where $\Phi_N^t$ defined by \eqref{defphitN} is generated by the particle system $(\ref{eq1},\ref{defG},\ref{defeqphi0},\ref{deff})$   and $\rho^t_N$ is the solution of the  $N$-body Liouville type one defined by $(\ref{liouvansatz},\ref{defG2},\ref{defeqphi},\ref{fansatz})$ with initial data $\rhoi$.

Applying Theorem  \ref{thmpropwas} with

\begin{eqnarray}
(\vlau_1([\rhoin]^{\leq s},X,V))_i&=& F_i(t,X,V)=\frac1N\sum_{j=1}^N\gamma(v_i-v_j,x_i-x_j)+\eta
\nabla_{x}\varphi^t(x_i)
+F_{ext}(x_i)  \nonumber\\
(\vlau_2([\rhoin]^{\leq s},X,V))_i&=& G_i(t,X,V)=\frac1N\sum_{j=1}^N\gamma(v_i-v_j,x_i-x_j)+\eta\nabla_{x}\Psi^t(x_i)
+F_{ext}(x_i)  \nonumber
\end{eqnarray} 
we get easily that, since the initial conditions are the same,
\begin{eqnarray}
&&\MKd((\Phi_N^t\#\rho^{in})_{N;1},(\rho^t_N)_{N:1})^2\label{truc1}\\
&\leq&
4\eta\frac1N\sum_{i=1}^N\int_0^t
\int_{\bR^{2dN}}
|
\nabla\varphi^s(x_i)-\nabla\Psi^s(x_i)
|^2(\Phi^s_N\#(\rho^{in})^{\otimes N})(dX,dV)
e^{\int_s^t\bar L_{\rhoi}(u)du}ds\nonumber
\end{eqnarray}
with
\be\label{deflptolv}
\bar L_{\rhoi}(u)=
2+2\left(
\sup_{\substack{i,l=1,\dots,N\\
(X,V)\in supp(
\rho^t_N)}}\Lip{(\gen)}^2_{(x_i-x_l,v_i-v_l)}
+2u\eta\Lip(\nabla\chi)+\Lip(F_{ext})\right)^2.
\ee

Therefore, we have to estimate
\begin{eqnarray}\label{zzero}
&&\int_{\bR^{2dN}}
|
\nabla\varphi^s(x_i)-\nabla\Psi^s(x_i)
|^2(\Phi^s_N\#(\rho^{in})^{\otimes N})(dX,dV)\nonumber\\
&=&
\int_{\bR^{2dN}}|\nabla\varphi^s(x_i^s(X,V))-\nabla\Psi^s(x_i^s(X,V))
|^2(\rho^{in})^{\otimes N}(dX,dV)
\end{eqnarray}
where we  have denoted
\be\label{defxi}
\Phi^t_N(X,V)
=:(x_1^t(X,V),\dots,x_N^t(X,V),v_1^t(X,V),\dots,v_N^t(X,V))
\ee
i.e. $x_i^t(X,V)$ is the $x_i$-component of $\Phi^t_N(X,V)$,

We first remark that, in \eqref{threeq}, 
$$
f(\cdot,X)=\chi *\mu_{Z},\ 
$$
where, for any $Z=(z_1,\dots,z_N)\in\bR^{2dN}$,  the empirical measure $\mu_Z$  is defined by
\be\label{defemp}
\mu_Z:=\tfrac1N\sum_{k=1}^N\delta_{z_k}
\ee
Therefore, by \eqref{solvedas},
$$
\nabla\varphi^s(x_i^s(X,V))
=
e^{-\eta s}\int_0^se^{(s-u)\Delta}\nabla\chi*\mu_{\Phi^u(X.V)}(x_i^s(X,V))du
$$
so that, denoting $\nabla_l,\ l=1,\dots,d,$ the $d$ components of the vector $\nabla$,
%
%
\begin{eqnarray}
&&\ \ \ \ \ \ \ \int_{\bR^{2dN}}| 
\nabla\varphi^s(x_i)-\nabla\Psi^s(x_i)
|^2(\Phi^t_N\#(\rho^{in})^{\otimes N})(dX,dV)
\label{coupe}\\
&=&e^{-2\kappa s}\int_{\bR^{2dN}}
\left|\int_0^se^{(s-u)\Delta}\nabla\chi*(\mu_{\Phi^u(X,V)}-(\rho^u_N)_{N;1})(x_i^s(X,V))du\right|^2
(\rho^{in})^{\otimes N}(dX,dV),\nonumber\\
&=&e^{-2\kappa s}\int_{\bR^{2dN}}\sum_{l=1}^d
\left|\int_0^se^{(s-u)\Delta}\nabla_l\chi*(\mu_{\Phi^u(X,V)}-(\rho^u_N)_{N;1})(x_i^s(X,V))du\right|^2
(\rho^{in})^{\otimes N}(dX,dV),\nonumber\\
&\leq&e^{-2\kappa s}\int_{\bR^{2dN}}\sum_{l=1}^d
\left|\int_0^s\|e^{(s-u)\Delta}\nabla_l\chi*(\mu_{\Phi^u(X,V)}-(\rho^u_N)_{N;1})\|_{L^\infty(\bR^d)}
du\right|^2
(\rho^{in})^{\otimes N}(dX,dV),\nonumber\\
&\leq&
e^{-2\kappa s}\sum_{l=1}^d\Lip(\nabla_l\chi)^2\int
|\int_0^s\MKd(\mu_{\Phi^u(X,V)},(\rho^u_N)_{N;1})du|^2
(\rho^{in})^{\otimes N}(dX,dV),\nonumber\\
&=&
e^{-2\kappa s}\Lip(\nabla\chi)^2\int
|\int_0^s\MKd(\mu_{\Phi^u(X,V)},(\rho^u_N)_{N;1})du|^2
(\rho^{in})^{\otimes N}(dX,dV),\nonumber\\
&=&
e^{-2\kappa s}\Lip(\nabla\chi)^2\int
|\int_0^s\MKd((\rho_{\Phi^u(X,V)})_{N:1},(\rho^u_N)_{N;1})du|^2
(\rho^{in})^{\otimes N}(dX,dV),\nonumber
\end{eqnarray}
where we have used Lemma \ref{systemat} for the second inequality and the following result for the last equality.
\begin{Lem}\label{deltempir}
Let $\rho_Z$ be defined by \eqref{ansatz} and $\mu_Z$  by \eqref{defemp}. Then
$$
\mu_Z=
(\rho_Z)_{N:1}.
$$
\end{Lem}
\begin{proof}
Let us recall that $\Sigma_N=\{\sigma:\{1,\dots,N\}\to\{1,\dots,N\}, \sigma\mbox{ one-to-one}\}$ so that $\#\Sigma_N=N!$. We have
\begin{eqnarray}
\int\dots\int \delta_{\sigma(Z)}dz_2\dots dz_N&=&
\int\dots \int \delta_{\sigma(Z)}
\prod_{l\neq\sigma(1)}dz_l\nonumber
\end{eqnarray}
Therefore
\begin{eqnarray}
\left(\frac1{N!}\sum_{\sigma\in\Sigma_N}\delta_{\sigma(Z)}\right)_{N:1}
&=&
\frac1{N!}\sum_{l=1}^N
\sum_{\substack{\sigma\in\Sigma_Nint\\
\sigma(l)=1}}
\int\dots \int \delta_{\sigma(Z)}
\prod_{l\neq\sigma(1)}dz_l
\nonumber\\
&=&
\sum_{l=1}^N\frac{\#\Sigma_{N-1}}{N!}\delta_{z_l}
=
\sum_{l=1}^N\frac{(N-1)!}{N!}\delta_{z_l}=\mu_Z.\nonumber
\end{eqnarray}
\end{proof}
By \eqref{liouvansatz}, $\rho_Z^s:=\rho_{\Phi^s(Z)}$ solves the  $N$-body Liouville type one defined by $(\ref{liouvansatz},\ref{defG2},\ref{defeqphi},\ref{fansatz})$ with initial data $\rhoi:=\rho_Z$. Therefore, by the Dobrushin estimate in Theorem \ref{thmliouvillevlasov}, one has
\be\label{hope}
\MKd((\rho_{\Phi^u(X,V)})_{N:1},(\rho^u_N)_{N;1})\leq2e^{\Gamma_N(u)}\MKd((\rho_Z)_{N:1},((\rhoi)^{\otimes N}_{N:1})
=2e^{\Gamma_N(u)}\MKd(\mu_Z,\rhoi),
\end{equation}
so that
\begin{eqnarray}
&&\int| (
\nabla\varphi^s(x_i)-\nabla\Psi^s(x_i))
|^2(\Phi^t_N\#(\rho^{in})^{\otimes N})(dX,dV)
\nonumber\\
&\leq&
4e^{-2\kappa s}\Lip(\chi)^2\left(\int_0^se^{\Gamma_N(u)}du\right)^2\int
\MKd(\mu_Z,\rhoi)^2
(\rho^{in})^{\otimes N}(dZ),\nonumber\\
&\leq&
4e^{-2\kappa s}\Lip(\chi)^2e^{\sup\limits_{u\leq s}\Gamma_N(u)}s^2
\int_{\bR^{2d}}(x^2+v^2)\rho^{in}(dx,dv)
C
\left\{
\begin{array}{ll}
N^{-\frac12}&d=1\\
N^{-\frac12}\log{N}&d=2\\
N^{-\frac1d}&d>2
\end{array}
\right.
\label{truc3},
\end{eqnarray}

 thanks to the following result by Fournier and Guillin:
\begin{Thm}[Theorem 1 in \cite{FG}]\label{forun}
Let $\rho\in\cP(\bR^{2d})$ satisfy
$$
\int_{\bR^{2d}} (x^2+v^2)\rho(dx,dv):=M_2(\rho)<\infty.
$$
and let $\mu_{(X,V)},\ (X,V)\in\bR^{2dN}$, be the empirical measure  defined by \eqref{defemp}.
Then
$$
\int_{\bR^{2dN}}\MKd(\mu_{(X,V)},\rho)^2\rho^{\otimes N}(dXdV)\leq
C_d(N)M_2(\rho),
$$
where
$${}
C_d(N):=C
\left\{
\begin{array}{ll}
N^{-\frac12}&d=1\\
N^{-\frac12}\log{N}&d=2\\
N^{-\frac1d}&d>2
\end{array}
\right.
$$
where $C$ depends only on $d$.
\end{Thm}

Therefore, we get by \eqref{truc1} and \eqref{truc3} the final result of this section.
\begin{Prop}\label{ptolv}
\begin{eqnarray}
\MKd((\Phi_N^t\#\rho^{in})_{N;1},
(\rho^t_N)_{N:1}
)
&\leq & 
\beta_{\rho^{in}}(t)\sqrt{C_d(N)}
\nonumber
\end{eqnarray}
with
\be\label{PtoLV}
\beta_{\rho^{in}}(t)^2=16\eta^2t^2\Lip(\nabla\chi)^2e^{t\sup\limits_{s\leq t} \bar L_{\rhoi}(s)}e^{\ 2\sup\limits_{s\leq t}\Gamma_N(s)}\int_{\bR^{2d}}(x^2+v^2)\rho^{in}(dx,dv).
\ee
\end{Prop}
\begin{Rmk}\label{rmkbarL}
Let us define $\bar\beta_{\rhoi}\geq \beta_{\rhoi}$ by
\be\label{defbarL}
\bar\beta_{\rhoi}:=16\eta^2t^2\Lip(\nabla\chi)^2e^{t\sup\limits_{s\leq t} \bar L_{\rhoi}(s)}e^{\ 2\sup\limits_{s\leq t}\Gamma_N(s)}
\sup_{(x,v)\in\ supp(\rhoi)}(x^2+v^2).
\ee
By its definition \eqref{deflptolv}, $\bar L_{\rhoi}(t)$ is, for fixed $t$, an increasing function of $supp(\rhoi)$ and, by Remark \ref{rmkgammaN}, $\Gamma_N$ is independent of $\rhoi$. Therefore $\bar \beta_{\rhoi}(t)$ is, for fixed $t$, an increasing function of $supp(\rhoi)$.
\end{Rmk}

\section{From Liouville-Vlasov to Vlasov
}\label{seclvtov}

In this section we estimate $\MKd((\rho^t_N)_{N:1},\rho^t)$, where $\rho^t_N$ is the solution of the  $N$-body Liouville type one defined by $(\ref{liouvansatz},\ref{defG2},\ref{defeqphi},\ref{fansatz})$  and $\rho^t$ is the solution of the   Vlasov system $(\ref{eq1V},\ref{defGV},
\ref{defeqphiV})$,  with initial data $(\rhoi)^{\otimes N}$ and $\rhoi$.

We first remark that $(\rho^t)^{\otimes N}$ solves the equation 
$$
\partial_t(\rho^t)^{\otimes N}+V\cdot\nabla_X(\rho^t)^{\otimes N}=\nabla_V.\cdot(\vlau_2([\rhoin]^{\leq t},X,V)(\rho^t)^{\otimes N}),
$$
with
$$
\vlau_2([\rhoin]^{\leq t},\cdot,\cdot):=\vla([\rhoi]^{\leq t},\cdot,\cdot)^{\otimes N}.
$$

Therefore, applying again Theorem  \ref{thmpropwas} taking this time
\begin{eqnarray}
(\vlau_1([\rhoin]^{\leq s},X,V))_i&=& G_i(t,X,V)=\frac1N\sum_{j=1}^N\gamma(v_i-v_j,x_i-x_j)+\eta\nabla_{x}\Psi^t(x_i)
+F_{ext}(x_i)   \nonumber\\
(\vlau_2([\rhoin]^{\leq s},X,V))_i&=& \vla([\rhoi]^{\leq t},x_i,v_i)=\gamma*\rho^t(x_i,v_i)+\eta\nabla_x\psi^t(x_i)+F_{ext}(x_i),  \nonumber
\end{eqnarray} 
we get easily that

\begin{eqnarray}
\MKd((\rho^t_1)_{N:1},(\rho^t_2)_{N:1})^2&\leq&
4\frac1N\sum_{i=1}^N\int_0^t
\int_{\bR^{2dN}}
\left(\left|\frac1N\sum_{j=1}^N\gen(x_i-x_j,v_i-v_j)-\gen*\rho^s(x_i,v_i)\right|^2\right.\nonumber\\
&&\left.{\color{white}\sum_{j=1}^N}+\eta^2|(
\nabla\psi^s(x_i)-\nabla\Psi^s(x_i))
|^2\right)(\rho^{s})^{\otimes N}(dX,dV)
e^{\int_s^t{\bar L_{\rhoi}}(u)du}ds,\nonumber
\end{eqnarray}
with the same factor $\bar L_{\rhoi}$ as in Section \ref{secptolv}, namely
\be\label{deflptolv2}
\bar L_{\rhoi}(u)=
2+2\left(
\sup_{\substack{i,l=1,\dots,N\\
(X,V)\in supp(
\rho^t_N)}}\Lip{(\gen)}^2_{(x_i-x_l,v_i-v_l)}
+2u\eta\Lip(\nabla\chi)+\Lip(F_{ext})\right)^2.
\ee
The first term in the integral has been estimated in \cite[Lemma 3.5, Section 3]{rt} and we get
\begin{eqnarray}\label{firsterm}
&&\int_{\bR^{2dN}}
\left|\frac1N\sum_{j=1}^N\gen(x_i-x_j,v_i-v_j)-\gen*\rho^s(x_i,v_i)\right|^2(\rho^{s})^{\otimes N}(dX,dV)\nonumber\\
&\leq&
\frac4N\sup_{(x,v),(x',v')\in supp(\rho^t)}|\gamma(x-x',v-v')|^2.\nonumber
\end{eqnarray}
It remains to estimate
\be\label{zun}
\int_{\bR^{2dN}}
|(
\nabla\psi^s(x_i)-\nabla\Psi^s(x_i))
|^2(\rho^{t})^{\otimes N}(dX,dV).
\ee

We have
\begin{eqnarray}
&&
|\nabla\Psi^s(x_i)-\nabla\psi^s(x_i)|^2\nonumber\\
&=&
e^{-2\kappa s}|\int_0^s(e^{(s-s')\Delta}\nabla\chi) *((\rho^{s'}_N)_{N;1}-\rho^{s'}))(x_i)ds'|^2\nonumber\\
&=&
e^{-2\kappa s}\int_0^sds'\int_0^sds"((e^{(s-s')\Delta}\nabla\chi) *((\rho^{s'}_N)_{N;1}-\rho^{s'})(x_i)((e^{(s-s")\Delta}\nabla\chi) *((\rho^{s"}_N)_{N;1}-\rho^{s"})(x_i).\nonumber
\end{eqnarray}
But, by Lemma \ref{systemat},
$$
\left|(e^{(t-s)\Delta}\nabla\chi) *((\rho^{s'}_N)_{N;1}-\rho^{s'})(x_i)\right|
\leq
\Lip{(\nabla\chi)}\MKd((\rho^{s'}_N)_{N;1},\rho^{s'}).
$$

Therefore
$$
|\nabla\Psi^s(x_i)-\nabla\psi^s(x_i)|^2
\leq
\Lip{(\nabla\chi)}^2e^{-2\kappa s}
\int_0^s\int_0^s
\MKd((\rho^{s'}_N)_{N;1},\rho^{s'})\MKd((\rho^{s"}_N)_{N;1},\rho^{s"})ds'ds",
$$
and 
we get 
\begin{eqnarray}\label{oufffdif}
\MKd((\rho^t_N)_{N;1},\rho^t)^2
&\leq&
\frac4N \int_0^t
\sup_{(x,v),(x',v')\in supp(\rho^t)}|\gamma(x-x',v-v')|^2
e^{\int_s^t\bar L_{\rhoi} (u)du}ds\nonumber\\
+\int_0^te^{\int_s^tL_{\rhoi} (u)du}&&
\int_0^sds'\int_0^sds''\MKd((\rho^{s'}_N)_{N;1},\rho^{s'})
\MKd\big(\rho^{s''}_N)_{N;1},\rho^{s''}\big)ds\nonumber\\
&:=&
\frac{C_{\rho^{in}}(t)}N
+\eta\Lip{\nabla\chi}^2\int_0^t
e^{-2\kappa s}e^{\int_s^t{\bar L_{\rhoi}} (u)du}ds\times\nonumber\\&&\int_0^sds'\int_0^sds''\MKd\big((\rho^{s'}_N)_{N;1},\rho^{s'}\big)
\MKd\big((\rho^{s''}_N)_{N;1},\rho^{s''}\big)\label{defct}
\end{eqnarray}
Let us define
$$
f(t)=\sup_{0\leq t'\leq t}\MKd\big((\rho^{t'}_N)_{N;1},\rho^{t'}\big)^2.
$$
We have, since by the definition \eqref{defct}, $C_{\rhoi}(t)$ is not decreasing,
\begin{eqnarray}
f(t)&\leq&
\frac{
C_{\rho^{in}}(t)}N
+\eta\Lip{\nabla\chi}^2\int_0^t
e^{-2\kappa s}e^{\int_s^t{\bar L_{\rhoi}} (u)du}ds\times\nonumber\\&&\int_0^sds'\int_0^sds''\MKd\big((\rho^{s'}_N)_{N;1},\rho^{s'}\big)
\MKd\big((\rho^{s''}_N)_{N;1},\rho^{s''}\big)\nonumber\\
&\leq&
\frac{C_{\rho^{in}}(t)}N
+\eta\Lip{\nabla\chi}^2\int_0^t
e^{-2\kappa s}e^{\int_s^t{\bar L_{\rhoi}} (u)du}s^2
f(s)ds.\nonumber
\end{eqnarray}
We conclude by the Gr\"onwall Lemma,
\begin{eqnarray}
\MKd((\rho^t_N)_{N;1},\rho^t)^2\leq f(t)&\leq&
\frac{
C_{\rho^{in}}(t)
}Ne^{ \eta\Lip{\nabla\chi}^2\int_0^t
e^{s\sup\limits_{u\leq s}\bar L_{\rhoi} (u)}
\tfrac{s^2}2ds}
:=
\frac{\alpha_{\rho^{in}}(t)^2}N,\label{defalpha}
\end{eqnarray}
%

and get the final result of this section.
\begin{Prop}\label{lvtov}
\be\label{LVtoV}
\MKd((\rho^t_N)_{N;1},\rho^t)
\leq 
\frac{\alpha_{\rho^{in}}(t)}{\sqrt N}.\nonumber
\ee
\end{Prop}
Out of $\alpha_{\rho^{in}},\beta_{\rho^{in}},
$  defined in \eqref{PtoLV}-\eqref{oufffdif} we define
\be\label{PtoVtau}
\tau_{\rho^{in}}(t)=\left(\alpha_{\rho^{in}}(t)+\beta_{\rho^{in}}(t)\sqrt C\right)^2,
\ee
where $C$ is the constant appearing   in Theorem \ref{forun}.
\vskip 1cm
One sees immediately that $C_{\rhoi}(t)$ defined by \eqref{defct} is, for fixed $t$, an increasing function of $supp(\rhoi)$.Therefore so is $\alpha_{\rhoi}(t)$ as defined by \eqref{defalpha} so that, if we define
\be\label{defbartau}
\bar\tau_{\rhoi}(t)=\left(\alpha_{\rho^{in}}(t)+\bar\beta_{\rho^{in}}(t)\sqrt C\right)^2,
\ee
where $C$ is the constant appearing   in Theorem \ref{forun}, we have
\be\label{bbar}
supp(\rhoi_1)\subset supp(\rhoi_2)\Rightarrow \tau_{\rhoi_1}\leq \bar\tau_{\rhoi_1}\leq \bar\tau_{\rhoi_2}.
\ee

\subsection*{Estimating $\tau_{\rhoi}(t)$}\ 

By the same type of arguments than in the proof of \cite[Corollary 2.6  ]{rt} one can easily estimate $\tau_{\rhoi}(t)$ by using the estimates of $\Phi^t_N$ established in Proposition \ref{prophitn} and Theorem \ref{thmphitn}, on the support of $\rho^t_N$ given by Theorem \ref{thmvlasov} and the support of $\rho^t$ in Theorem \ref{thmvlasov}. We omit the details here.


We get, for some time independent constant $C$ depending explicitly on and only on  $\Lip(\gen,\Lip(\chi),\Lip(\nabla \chi)$ and $|supp(\rhoi)|$,
\be\label{estau}
\tau_{\rhoi}(t)\leq e^{e^{Ct}}.
\ee
\vskip 1cm
\vskip 1cm

\section{Hydrodynamic limit 
}\label{mono}

{
The hydrodynamic limit of Cucker-Smale models has provided up to now a large litterature, whose exhaustive quotation is beyond the scope of the present paper. We refer to \cite{carr3} and the large bibliography therein. In \cite{carr3},  the corresponding Euler equation is derived for Cucker-Smale systems with friction, using the empirical measures formalism and in a modulated energy topology.

Our approach  and results are  different: we consider generalizations of frictionless Cucker-Smale models, coupled to chemotaxis  through a diffusive interaction, for large numbers $N$ of particles and we provide explicit rates of convergences in the quadratic Wasserstein metric 
towards Euler type equations.

Our result happens to be a simple corollary of our main result Theorem \ref{mainthm} in the case where $\rho^{in}$ is monkinetic, i.e.
$$
\rho^{in}(x,v)=\mu^{in}(x)\delta(v-u^{in}(x))
$$
thanks to the following  result: the monokinetic form is preserved by the    Valsov equation \eqref{eq1V} and  the solution is furnished by the solution of a 
Euler type equation.

}

\begin{Thm}\label{hydro}
Let $\mu^t,u^t,\psi^t$ be a solution to the following system
\be\label{systhydro}
\left\{
\begin{array}{l}
\partial_t\mu^t+\nabla(u^t \mu^t) = 0\\
\partial_t(\mu^tu^t)+
\nabla(\mu^t(u^t)^{\otimes 2}) = 
\mu^t\int\gamma(\cdot-y,u^t(\cdot)-u^t(y))\mu^t(y)dy+
\eta\mu^t\nabla\psi^t
+\mu^tF\\
\partial_s\psi^s = D\Delta\psi-\kappa\psi +
\chi*\mu^s,\ s\in[0,t],
\\ 
(\mu^0,u^0,\psi^0) = (\mu^{in},u^{in},\psi^{in})\in H^s,\ s>\tfrac d2+1.
\end{array}
\right.\nonumber
\ee
where 
$
\mu^t, u^t\in C([0,t];H^s)\cap C^1([0,T];H^{s-1}),\ 
\psi^t  \in C([0,t];H^{s})\cap C^1([0,T];H^{s-2}) \cap L^2(0,T; H^{s+1})$
\footnote{
{We suppose this regularity because it is somehow standard for mixed hyperbolic-parabolic  systems (see \cite[Theorem 2.9  p. 34
]{dk}, one certainly could low it down.}}.

Then $\rho^t(x,v):=\mu^t(x)\delta(v-u^t(x))$ solves 
the following system
\be\label{vlasobhydro}
 \left\{
\begin{array}{l}
\partial_t\rho^t+v\cdot\nabla_x\rho^t=\nabla_v(
\vla(t,x,v)\rho^t),\ 
\\ 
\vla(t,x,v)=\gamma(x,v)*\rho^t+\eta\nabla_x\psi^t(x)+F_{ext}(x),\\ 
\partial_s\psi^s(z)=D\Delta_z\psi-\kappa\psi+g(z,\rho^s),\ 
\psi^0=\psi^{in},\\
\rho^{0}(x,v)=\mu^{in}(x)\delta(v-u^{in}(x)).
\end{array}
\right.\nonumber
\ee

\end{Thm}
\begin{proof}
%
{When $\eta=0$, the derivation of \eqref{vlasobhydro} out of \eqref{systemat} is standard, see e.g. \cite{fk} and \cite[Section 1.2]{carr3}. The addition of the term $\eta\nabla_x\psi^t$ is a straightforward generalization.}
\end{proof} 

Notice that the Euler system in Theorem \ref{hydro} can be compared to the so-called Preziosi model of vasculogenesis, where  the nonlocal term  in our Euler system is replaced  by a phenomenological  pressure gradient local in the density, see \cite{prez}.
\section{Estimates on the solution of the particle system 
}\label{estpart}

Global existence and uniqueness for the system $(P)$ has been proved when $\gen$ is exactly the Cucker-Smale field in \cite{dicostanzo}. It is straightforward to adapt the proofs to the case of a general $\gen$
 satisfying the hypothesis of the present paper. This situation is anyway fully  included in \cite[Theorem 6]{mp} and we have the following result.
 \begin{Thm}\label{thmliouvillevlasov}
Let $\Lip(\gamma),\Lip(\nabla\chi)<\infty$ and let $Z^{in}\in\bR^{2dN}$.
Then,  for any $N\in\bN$,the Cauchy problem

$$
\begin{array}{cl}
(P)& \left\{ 
\begin{array}{l}
\dot {x_i}=v_i,\  
\dot {v_i}=F_i(t,X(t),V(t)),\ 
(X(0),V(0))=Z(0)=Z^{in}\in\bR^{2dN} 
\\ \\
F_i(t,Y,W)=\frac1N\sum\limits_{j=1}^N\gamma(w_i-w_j,y_i-y_j)+\eta\nabla_{z}\varphi^t(z)|_{z=y_i}
+F_{ext}(y_i),\\ \\
\partial_s\varphi^s(z)
=D\Delta_z\varphi-\kappa\varphi +f(z,X(s)),\ s\in[0,t],\\ \\
f(z,X)=\frac1N\sum\limits_{j=1}^N\chi (z-x_j).
\end{array} 
\right.
\nonumber
\end{array}
$$
has a unique global solution in $C^0(\bR,\bR^{2dN})$.
\end{Thm}
 Estimates on the solution of $(P)$ can be easily obtained by the same kind of proof  that in \cite[Appendix A]{rt}. We get the following result.
 \begin{Prop}\label{prophitn}Let $\gamma_0=\Lip(\gen)+\Lip(\nabla\chi)$. Then, for all $t\in\bR$, the solution of $(P)$ satisfies
 \begin{eqnarray}\label{estpart}
|v_i(t)|&\leq& 
\max_{j=1,\dots,N}|v_j(0)|e^{2\gen_0t},\ \ \ i=1,\dots,N,\nonumber\\
||x_1(t)|-|x_i(0)||
&\leq &
\max_{j=1,\dots,N}|v_j(0)|\frac{e^{2\gen_0t}-1}{2\gen_0},\ \ \ i=1,\dots,N.\nonumber
\end{eqnarray}
 \end{Prop}
 
 Finally, we will need the following estimate on the derivative of the flow generated by the system $(P)$.
 \begin{Thm}\label{thmphitn}
 Let $z_i(t)=x_i(t),v_i(t),\ i=1,\dots,N$ be the solution of $(P)$ with initial date $z_i(0)=z^{in}_i$. Then, for all $T\in\bR$,
 $$
 \sup_{t\leq T}\left|\frac{\partial z_i(t)}{\partial {z_j^{in}}}\right|\leq e^{(\gen_1 +\gen_2T)t}, \  i,j=1,\dots,N
 $$
 with 
 $\gen_1=\Lip(\gen),\gen_2=\Lip(\nabla\chi)$.
 
 In other words,
 $$
 \sup_{t\leq T}\|d\Phi^t_N\|_\infty\leq e^{(\gen_1+\gen_2 T)t}.
 $$
 \end{Thm}
 \begin{proof}
 One easily get that, for each $i,j=1,\dots,N$,
 \begin{eqnarray}
 \left|\partial_t\frac{\partial z_i(t)}{\partial {z_j^{in}}}\right|
 &\leq&
 \Lip(\gen)\frac1N\sum_{l=1}^N\left|\frac{\partial z_l(t)}{\partial {z_j^{in}}}\right|+\eta\Lip(\nabla\chi)
 \frac1N\sum_{l=1}^N\int_0^t\left|\frac{\partial z_l(s)}{\partial {z_j^{in}}}\right|ds\nonumber
 \end{eqnarray}
 Therefore, since the right hand-side of the preceding equality doesn't depend on $i$,
 \begin{eqnarray}
 \left|\partial_t\sum_{i=1}^N\frac{\partial z_i(t)}{\partial {z_j^{in}}}\right|
 &\leq&
 \Lip(\gen)\sum_{l=1}^N\left|\frac{\partial z_l(t)}{\partial {z_j^{in}}}\right|+\eta\Lip(\nabla\chi)
 \int_0^t\sum_{l=1}^N\left|\frac{\partial z_l(s)}{\partial {z_j^{in}}}\right|ds\nonumber
 \end{eqnarray}
 so that, since 
 $$
 \sum_{i=1}^N\frac{\partial z_i(0)}{\partial {z_j^{in}}}=
\sum_{i=1}^N\frac{\partial z_i^{in}}{\partial {z_j^{in}}}= 
\sum_{i=1}^N\delta_{i,j}= 
 1,
 $$
 \begin{eqnarray}
 \left|\sum_{i=1}^N\frac{\partial z_i(t)}{\partial {z_j^{in}}}-1\right|
 &\leq&
 \Lip(\gen)\int_0^t\sum_{l=1}^N\left|\frac{\partial z_l(u)}{\partial {z_j^{in}}}\right|du
 +\eta\Lip(\nabla\chi)
 \int_0^t\int_0^u\sum_{l=1}^N\left|\frac{\partial z_l(s)}{\partial {z_j^{in}}}\right|dsdu\nonumber\\
 &\leq &
 (\Lip(\gen)+T\eta\Lip(\nabla\chi))
 \int_0^t\sum_{l=1}^N\left|\frac{\partial z_l(u)}{\partial {z_j^{in}}}\right|du\nonumber
 \end{eqnarray}
 and, by Gr\"onwall Lemma,
 $$
 \left|\frac{\partial z_i(t)}{\partial {z_j^{in}}}\right|
 \leq \left|\sum_{i=1}^N\frac{\partial z_i(t)}{\partial {z_j^{in}}}\right|
 \leq e^{\gen_1t+\gen_2 Tt}.
 $$
 \end{proof}
%
%

\section{Existence, uniqueness  and Dobrushin stability  for the Vlasov system
}\label{euv}
In this section we study the Vlasov system $(V)$ and prove the properties that satisfy its solution necessary for the proof of our main result.
\begin{Thm}\label{thmvlasov}
Let $\Lip(\gamma),\Lip(\nabla\chi)<\infty$ and let $\vla^{in}\in\cP_c(\bR^{2d})$, the set of compactly supported probability meausres.
Then the Cauchy problem

$$
(V) \left\{
\begin{array}{l}
\partial_t\rho^t+v\cdot\nabla_x\rho^t=\nabla_v(
\vla(t,x,v)\rho^t),\ \rho^0=\rho^{in}
\\ \\
\vla(t,x,v)=\gamma*\rho^t(x,v)+\eta\nabla_x\psi^t(x)+F_{ext}(x),\\ \\
\partial_s\psi^s(z)=D\Delta_z\psi-\kappa\psi+g(z,\rho^s),\ \psi^0=\macphin.
\end{array}
\right.\nonumber
$$
{has a unique solution $t\to
\begin{pmatrix}
{\rho^t}\\{\psi^t}\end{pmatrix}$ in $C^0(\bR,\cP_c(\bR^{2d})\times 
W^{1,\infty}(\bR^d)
)
$.
}

Moreover, if $\rhoi$ is supported in the ball $B(0,R^0)$ of $\bR^{2d}$ centered at the origin of radius $R^0$, $\rho^t$ is supported in $B(0,R^t)$ with
$$
R^t=e^{(Lip(\gen)+\|F_{ext}\|_{L^\infty(\bR^d)}+\eta\Lip(\chi))t}\left(R^0+\Lip(\gen)+\|F_{ext}\|_{L^\infty(\bR^d)}+\eta\Lip(\chi)\right).
$$

Finally, if $\rho_1^t,\rho_2^t$ are the solutions of $(V)$ with initial conditions $\rhoi_1,\rhoi_2$, then the following Dobrushin  type estimate holds true
$$
\MKd(\rho_1^t,\rho_2^t)^2\leq 2e^{\Gamma(t)}\MKd(\rhoi_1,\rhoi_2)^2
$$
where $\Gamma(t)$ is given below by \eqref{defgam}.
\end{Thm}

\begin{proof}
The proof of the existence of a solution will follow closely the strategy of the proof of Theorem 2.3 in \cite[Appendix A]{prt}. The main difference is that 
$\vla$ is 
not only
non-local in space as in \cite{prt}, it is also non-local in time as $\vla([\rhoi]^{\leq t},x,v)$ involves the whole history of the solution $\{\rho^s,0\leq s\leq t\}$. 

We will first need the following Lemma

\begin{Lem}\label{commeprt}
For any $T\geq 0$, there exist $L'.M',K'<\infty$ such that, for any $t,t_1,t_2\leq T$, $z,z'\in\bR^{2d}$ and any $\rho^{in},\rho^{in}_1,\rho^{in}_2\in\cP(\bR^{2d})$,,
\begin{eqnarray}
\|\vlau(\hist{\rho^{in}}{t},z)-
\vlav{t}{z'}{\rhoi}\|&\leq&L'\|z-z'\|,\ \nonumber\\
\|\vla(\hist{\rho^{in}}{t},z)\|&\leq & M'(1+\|z\|)\nonumber\\
\|\vlau(\hist{\rho^{in}}{t_1},z,\hist{\rho^{in}_1}{t_1})-
\vlav{t_2}{z}{\rhoi_2}\|&\leq&
K' \sup_{
s\leq \min(t_1,t_2)}\MKu(\rho_1^{s},\rho_2^{s})+\eta\|\nabla\chi\|_{L^\infty
}|t_1-t_2|
.\nonumber
\end{eqnarray}
where $\rho_1^t,\rho_2^t$ are the solutions of $(V)$ with initial conditions $\rhoi_1,\rhoi_2$.

Here $\MKu$ is the Wasserstein distance of order $1$  whose definition is recalled in Definition \ref{defwasun}:
\end{Lem}
The proof is immediate with $L'=\Lip(\gen)+\Lip(F_{ext})+T\eta\Lip(\nabla\chi)$, $M'=\Lip(\gen)+\|F_{ext}\|_{L^\infty(\bR^d)}+\eta\Lip(\chi)+\Lip(\nabla\macphin)$ and $K'=\Lip(\gen)+\eta\Lip(\chi)$.
\vskip 1cm
Let us fix $T>0$. For $k\in\bN$ we define $\ta_k=T2^{-k}$. 

Let $\rho^t_k $ be defined by $\rho^{t=0}_k=\rho^{in}$ and, for $l=0,\dots,2^k-1,\ u\in[0,\tau_k)$,

$$
(V_k) \left\{
\begin{array}{l}
\partial_u\rho_k^{l\ta_k+u}(x,v)+v\cdot\nabla_x\rho_k^{l\ta_k+u}(x,v)=
\nabla_v\cdot\vla_k(l\ta_k,x,v)\rho_k^{l\ta_k+u}(x,v)
\\ \\
\vla_k(l\ta_k,x,v)=\gamma(x,v)*\rho_k^{l\ta_k}+\eta\nabla_x\psi_k(l\ta_k,x)+F_{ext}(x),\\ \\
\partial_s\psi_k(s,z)=D\Delta_z\psi_k-\kappa\psi_k+g(z,\rho_k^{s}),\ 0\leq s\leq l\ta_k,\ \psi_k^0=\macphin.
\end{array}
\right.\nonumber
$$
with
\be\label{defbv}
\vla_k(t,z)=
\vlavk{t}{z}{\rhoi}
\ee

Note that we have obviously the following corollary of Lemma \ref{commeprt}.
\begin{Cor}\label{commeprtk}
For any $T\geq 0$, 
$z,z'\in\bR^{2d}$ and $k\in\bN$, one has, with the same constant $L'.M',K'$ than in Lemma \ref{commeprt},
\begin{eqnarray}
\|\vla_k(t,z)-\vla_k(t,z')\|&\leq&L'\|z-z'\|,\nonumber\\
\|\vla_k(t,z)\|&\leq & M'(1+\|z\|)\nonumber,
\end{eqnarray}
Moreover, if $\vla_k(t,z),\rho_k^t$ satisfies $(V_k)$ with $\rho_k^{t=0}=\rho^{in}$ and $\vla_k'(t,z),{\rho'}^{t}_k$ satisfies $(V_k)$ with $\rho_k^{t=0}={\rho}'^{in}$, then
$$
\|\vla_k(t,z)-\vla_k'(t',z)\|_{L^\infty(\bR;C^0(\bR^{2d})}\leq K' \sup_{0\leq s\leq \min(t,t')}\MKu(\rho_k(s),\rho_k'(s))+K''|t-t'|.
$$
\end{Cor}

We first show that the support of the sequence $\rho_k^t$ is equibounded. 

One easily checks that,since $\rho^{in}$ is compactly supported, so is $\rho_k^t$ for all $k,t$ by construction.
 So  $supp(\rho_k^t)\subset B(0,R_k^t)$ for some $R_k(t)$. One can estimate $R_k^t$ as follows.
 
 \begin{eqnarray}
 supp(\rho_k^{l\tau_k})\subset B(0,R_k^{l\tau_k})&\Rightarrow&
 \|
 \vla_k(t,z)\|_\infty=\|
\vlavk{t}{z}{\rhoi}\|_\infty\leq M'(1+R_k^{l\tau_k})\nonumber\\
&\Rightarrow&
supp(\rho_k^{l\tau_k+u})\subset
B(0,R_k^{l\tau_k}+uM'(1+R_k^{l\tau_k})),\ u\in[0,\tau_k]\nonumber\\
&\Rightarrow&
supp(\rho_k^{(l+1)\tau_k})\subset
B(0,(1+\tau_k)R_k^{l\tau_k}+\tau_kM').\nonumber
\end{eqnarray}
Therefore, one can choose $R_k^{l\tau_k}$ satisfying
\begin{eqnarray}
R_k^{l\tau_k}&\leq&
(1+M'\tau_k)R_k^{(l-1)\tau_k}+\tau_kM'\nonumber\\
&\leq&
(1+M'\tau_k)^2R_k^{(l-2)\tau_k}+\tau_kM'(1+(1+M'\tau_k))\nonumber\\
&\leq&
(1+M'\tau_k)^{l}R^{0}
+M'(((1+M'\tau_k)^l-1)\nonumber\\
&\leq&
(1+M'T2^{-k})^{2^k}R^{0}
+M'((1+M'T2^{-k})^{2^k}-1)\leq e^{M'T}(M'+R^0):=R^T.\nonumber
\end{eqnarray}
Here $R^0$  is such that $supp(\rho_k^0:=\rho^{in})\subset B(0,R^0)$.

Hence the sequences $(\rho_k^t)_{k\in\bN}$ are  compactly supported in $B(0,R^T)$ uniformly in $k$  for all $t\in[0,T]$. Therefore there are tight for all $t\in[0,T]$. By Prokhorov's Theorem, this is equivalent to the compactness
of $(\rho_k^t)_{k\in\bN}$ with respect to the weak topology of probability measures (i.e., in
the duality with $C_b(R^{2d})$, the space of bounded continuous functions). Hence, up to extracting
a $t$-dependent subsequence, 
$\rho_k^t\to\rho_*^t$ weakly. By 
Lemma \ref{lemfund}, this convergence is equivalent to the convergence with respect to the distance $\MKu$, so that
we just proved that
\be\label{convponc}
\MKu(\rho_k^t,\rho_*^t)\to 0\mbox{ as }k\to\infty\mbox{ for all }t\in[0,T].
\ee
\vskip 1cm
By \cite[Proposition A.1 2]{prt} and the first inequality in Lemma \ref{commeprt}, we get that, for $l=0,\dots,2^k-1,\ s\in[0,\tau_k)$,
$$
\MKu(\rho_k^{l\ta_k},\rho_k^{l\ta_k+s})\leq sL'.
$$
Hence, by the triangle inequality,
\be\label{equik}
\MKu(\rho_k^t,\rho_k^{t'})\leq L'|t-t'|, \forall t,t'\in\bR.
\ee
Therefore, since $L'$ and $\rho_k^{t=0}$ don't depend on $k$, the sequence $\rho_k^t$ is equi-Lipschitz continuous with respect to
$\MKu$. This implies, by the triangular inequality again, that, for all $t,t'\in[0,T],k\in\bN$,
\begin{eqnarray}
\MKu(\rho_*^t,\rho_*^{t'})&\leq&
\MKu(\rho_*^t,\rho_k^t)+\MKu(\rho_k^t,\rho_k^{t'})+\MKu(\rho_k^{t'}
,\rho_*^{t'})\nonumber\\
&\leq&L'|t-t'|+
\MKu(\rho_*^t,\rho_k^t)+\MKu(\rho_k^{t'}
,\rho_*^{t'})
\to L'|t-t'|\mbox{ as }k\to\infty.\nonumber
\end{eqnarray}
Therefore $\rho_*^t$ is $L'$-Lipschitz and, in particular,
$$
\rho_*^t\in C_0([0,T],\cP_c(\bR^{2d})).
$$
%
%
%

What is left is to prove that $\rho^t_*$ solves $(V)$ and that the solution of $(V)$ is unique.
\vskip 0.5cm

In order to do so, we first need to prove that
\be\label{cauchy}
\sup_{t\in[0,T]}\MKu(\rho^t_k,\rho^t_*)\to 0\mbox{ as  }k\to\infty.
\ee

Using the standard Ascoli-Arzel\`a approach, we take the  sequence $\{q_i,i=1,\dots,\infty\}$ of all rationals in $[0,T]$ ordered by any fixed order.

From \eqref{convponc} we have that
$$
\MKu(\rho^{q_1}_{k_1},\rho^{q_1}_*)\to 0.
$$
along a subsequence $k_1$.

Take now a subsequence $k_2$ of $k_1$ such taht
$$
\MKu(\rho^{q_2}_{k_2},\rho^{q_2}_*)\to 0.
$$
and iterate such as to obtain a sequence $k_m$ such that
$$
\MKu(\rho^{q_l}_{k_m},\rho^{q_l}_*)\to 0, \mbox{ for all }l\leq m.
$$

In a standard way, we can extract a diagonal subsequence $k_k$ such that
$$
\MKu(\rho^{q_k}_{k_k},\rho^{q_k}_*)\to 0.
$$

We prove now that, for all $t\in[0,T]$, the sequence of densities $\rho^t_{k_k}$ is  Cauchy in the topology of $C_0([0,T],(\cP_c,{\MKu}))$, which implies \eqref{cauchy}. 

We have that, for any $q\in\bQ$,
\begin{eqnarray}
\MKu(\rho^t_{k_k},\rho^t_{l_l})&\leq&
\MKu(\rho_{k_k}^t,\rho^{q_i}_{k_k})+
\MKu(\rho_{k_k}^{q},\rho^{q}_{l_l})+
\MKu(\rho_{l_l}^{q},\rho^{t}_{l_l})\nonumber\\
&\leq&2L'|q-t|
+
\MKu(\rho_{k_k}^{q},\rho^{q}_{l_l})\nonumber
\end{eqnarray}
For any $\epsilon$, taking  $q$ such that $2L'|q-t|\leq\epsilon/2$ and $k,l$ large enough such that $\MKu(\rho_{k_k}^{q},\rho^{q}_{l_l})\leq\epsilon/2$, we get $MKu(\rho^t_{k_k},\rho^t_{l_l})\leq\epsilon$.
\vskip 0.5cm
We now return to proving that $\rho^t_*$ is a solution of $(V)$, it suffices to prove that
$$
(V^*) \left\{
\begin{array}{l}
\int_0^T\int_{\bR^{2d}}(\partial_tf+v\cdot\nabla_xf-\nabla_v f\cdot\vlau(t,z,\rho_*^{\leq t})\rho_*^t(dZ)dt=0
\\ \\
\vlau(t,(x,v),\rho_*^{\leq t})=\gamma(x,v)*\rho^t_*+\eta\nabla_x\psi^t(x)+F_{ext}(x),\\ \\
\partial_s\psi^s(z)=D\Delta_z\psi-\kappa\psi+g(z,\rho^s_*).
\end{array}
\right.\nonumber
$$
for each $f\in C_c^\infty([0,T]\times\bR^{2d})$. 

In $(V^*)$ we have used the notation $\rho^{\leq t}=\rho^s|_{s\leq t}$ for any function $t\in\bR\to\rho^t\in\cP(\bR^{2d})$.

By construction, we have
$$
\sum_{l=0}^{2^k-1}\int_{l\ta_k}^{(l+1)\ta_k}
\int_{\bR^{2d}}(\partial_uf(u,z)+v\cdot\nabla_xf-\nabla_v f\cdot\vla_k(l\ta_k,z,\rho_k^{\leq l\ta_k})\rho_k^u(dz)du=0
$$
for every $k\in\bN$.

The equation
\be\label{prum}
\int_0^T\int_{\bR^{2d}}(\partial_tf+v\cdot\nabla_xf-\nabla_v f\cdot\vlau(t,z,\rho_*^{\leq t})\rho_*^t(dZ)dt=0
\ee
will be proven through the three following limts:

\be\label{lim1}
\lim_{k\to\infty}
\int_0^T\int_{\bR^{2d}}(\partial_tf+v\cdot\nabla_xf)
(\rho_*^t-\rho_k^t)(dz)dt=0
\ee
\be\label{lim2}
\lim_{k\to\infty}
\sum_{l=0}^{2^k-1}\int_{l\ta_k}^{(l+1)\ta_k}
\int_{\bR^{2d}}\nabla_v f\cdot\left(\vlau(l\ta_k,z,\rho_k^{\leq l\ta_k})-\vlau(u,z,\rho_*^{\leq l\ta_k})\right)\rho_*^u(dz)du=0
\ee
\be\label{lim3}
\lim_{k\to\infty}
\sum_{l=0}^{2^k-1}\int_{l\ta_k}^{(l+1)\ta_k}
\int_{\bR^{2d}}\nabla_v f\cdot\vlau(l\ta_k,z,\rho_*^{\leq l\ta_k})\left(\rho_k^u-\rho_*^u\right)(dz)du=0
\ee
To prove \eqref{lim1} and \eqref{lim3} we remark that, since $f\in C_c^\infty([0,T]\times\bR^{2d})$ and by the Lipschitz property of $\vlau(l\ta_k,z,\rho_*^{\leq l\ta_k})$ we have, by 
\eqref{rabin}, that the absolute value of the right hand side of \eqref{lim1} satisfies, thanks to \eqref{cauchy},
\begin{eqnarray}
&&|\lim_{k\to\infty}
\int_0^T\int_{\bR^{2d}}(\partial_tf++v\cdot\nabla_xf)
(\rho_*^t-\rho_k^t)(dz)dt|\nonumber\\
&&\leq
\lim_{k\to\infty}T(\Lip(\partial_tf)+\Lip(v\cdot\nabla_xf))
\sup_{t\leq T}\MKu(\rho_k^t,\rho_*^t)=0,\nonumber
\end{eqnarray}
and the absolute value of the right hand side of \eqref{lim3} satisfies, thanks again to \eqref{cauchy},
\begin{eqnarray}
&&|\lim_{k\to\infty}
\sum_{l=0}^{2^k-1}\int_{l\ta_k}^{(l+1)\ta_k}
\int_{\bR^{2d}}\nabla_v f\cdot\vlau(l\ta_k,z,\rho_*^{\leq l\ta_k})\left(\rho_k^u-\rho_*^u\right)(dz)du
|\nonumber\\
&&\leq
\lim_{k\to\infty} T\Lip(\nabla f)|supp(f)|M'
\sup_{t\leq T}\MKu(\rho_k^t,\rho_*^t)=0,\nonumber
\end{eqnarray}
Let us finally prove \eqref{lim2}. By the third inequality in Lemma \ref{commeprt} and the $L'$-Lipschitz continuity of $\rho_k$ we get that, for each $u\in[l\ta_k, (l+1)\ta_k]$,
\begin{eqnarray}
\left|\vlau(l\ta_k,z,\rho_k^{\leq l\ta_k})-\vlau(u,z,\rho_*^{\leq l\ta_k})\right|&\leq &
K'\sup_{0\leq s\leq l\ta_k}\MKu(\rho_k^{l\ta_k},\rho_*^{l\ta_k})+K''\frac T{2^k}\nonumber
\end{eqnarray}
so that
\begin{eqnarray}
&\lim\limits_{k\to\infty}
\left|\sum\limits_{l=0}^{2^k-1}\int_{l\ta_k}^{(l+1)\ta_k}
\int_{\bR^{2d}}\nabla_v f\cdot\left(\vlau(l\ta_k,z,\rho_k^{\leq l\ta_k})-\vlau(u,z,\rho_*^{\leq l\ta_k})\right)\rho_*^u(dz)du\right|&\nonumber\\
&
\lim\limits_{k\to\infty} T\|\nabla f\|_{L^\infty(\bR^{2d})}
\left(
\sup\limits_{0\leq t\leq T}\MKu(\rho_k^t,\rho_*^t)+K''\frac T{2^k}
\right)=0&.\nonumber
\end{eqnarray}

The proof of the uniqueness of the solution of $(V)$ is 
%
obviously a consequence of 
the Dobrushin stability result that we will prove now.

Take two initial conditions $\rhoi_l,l=1.2$.
By Theorem \ref{thmpropwas}, we have that
\begin{eqnarray}
&&\MKd(\rho^t_1,\rho^t_2)^2\nonumber\\
&\leq& e^{\int_0^tL_1(s)ds}
\MKd(\rhoi_1,\rhoi_2)^2\nonumber\\
&+&
\int_0^t\int_{\bR^{2dN}}
2|
|\vlav{s}{Z}{\rhoi_1}-
\vlav{s}{Z}{\rhoi_2}|^2
(\rho^t_1)^{\otimes N}(dX,dV)e^{\int_s^tL_1(u)du}ds\nonumber\\
&\leq&
e^{\int_0^tL_1(s)ds}
\MKd(\rhoi_1,\rhoi_1)^2
+
\int_0^t
2
\|
|\vlav{s}{Z}{\rhoi_1}-
\vlav{s}{Z}{\rhoi_2}\|^2_\infty
e^{\int_s^tL_1(u)du}ds\nonumber\\
&\leq&
e^{\int_0^tL_1(s)ds}
\MKd(\rhoi_1,\rhoi_2)^2
+
2
\int_0^t
e^{\int_s^tL_1(u)du}(K')^2
\sup_{u\leq s}\MKd(\rho^u_1,\rho^u_2)
^2ds\
\nonumber
\end{eqnarray}
by the third inequality in Lemma \ref{commeprt}.

Here, by Theorem  \ref{thmpropwas}, 
$$
L_1(t)=
2(1+ \sup_{(X,V)\in\ supp(\rho_1^t)}
(\Lip(\vlau
([\rhoi_1]^{\leq t},X,V)
)_{(X,V)})^2.
$$

and by the first inequality in Lemma \ref{commeprt},
\begin{eqnarray}\label{lvlasov}
L_1(t)&=&2+
2(L')^2.
\end{eqnarray}

Therefore, for any $T\geq 0$,
\begin{eqnarray}
&&\sup_{t\leq T}
\MKd(\rho^t_1,\rho^t_2)^2\nonumber\\
&\leq&\sup_{t\leq T}
e^{\int_0^tL(s)ds}
\MKd(\rhoi_1,\rhoi_2)^2
+
2\sup_{t\leq T}
\int_0^t
e^{\int_s^tl(u)du}(K')^2
\sup_{u\leq s}\MKd(\rho^u_1,\rho^u_2)
^2ds\nonumber\\
&\leq&
e^{\int_0^TL(s)ds}
\MKd(\rhoi_1,\rhoi_2)^2
+
2
\int_0^T
e^{\int_s^Tl(u)du}(K')^2
\sup_{u\leq s}\MKd(\rho^u_1,\rho^u_2)
^2ds.\nonumber
\end{eqnarray}
By the Gr\"onwall Lemma, we get immediately
$$
\MKd(\rho^t_1,\rho^t_2)^2
\leq e^{\Gamma(t)}
\MKd(\rhoi_1,\rhoi_2)^2
$$ 
with
\be\label{defgam}
\Gamma(t):=
t\left(
2+2(L')^2+(2K')^2e^{2t)(1+(L')^2)}\right)
\ee
Theorem \ref{thmvlasov} is proved.
\end{proof}

\section{Existence, uniqueness and Dobrushin estimate for the Liouville-Vlasov system
}\label{eulv}
In this section we study the Liouville-Vlasov system $(LV)$ in the same spirit than in the previous section devoted to the Vlasov system.
\begin{Thm}\label{thmliouvillevlasov}
Let $\Lip(\gamma),\Lip(\nabla\chi)<\infty$ and let $\vla^{in}\in\cP_c(\bR^{2d})$, the set of compactly supported probability meausres.
Then, for every $N\in\bN$, the Cauchy problem

$$
(LV)\left\{
\begin{array}{l}
\partial_t\rho^t_N+V\cdot\nabla_X\rho^t_N=
 \sum\limits_{i=1}^N\nabla_{v_i}\cdot G_i\rho^t_N,\ \rho^o_N=
 \rhoi_N\\ \\
G_i(t,Y,W)=\frac1N\sum\limits_{j=1}^N\gamma(w_i-w_j,y_i-y_j)+\eta\nabla_{z}\Psi^t(z)|_{z=y_i}
+F_{ext}(y_i),
\\ \\
\partial_s\Psi^s(z)=D\Delta_z\Psi-\kappa\Psi +g(z,\rho^s_{N;1}),\ s\in[0,t],\ \Psi^0=\macphin
\\ \\
g(z,\rho^s_{N;1})=
\int_{\bR^{2d}}\chi(z-x)\rho^s_{N;1}(x,v)dxdv
\end{array}
\right.
$$
has a unique solution $t\to
\begin{pmatrix}
{\rho^t_N}\\{\Psi^t}\end{pmatrix}$ in $C^0(\bR,\cP_c(\bR^{2dN})\times 
W^{1,\infty}(\bR^d)
)$.

Moreover, if $
\rhoi_N$ is supported in the ball $B(0,R^0)$ of $\bR^{2dN}$ of radius $R^0$, $\rho^t_N$  centered at the origin is supported in $B(0,R^t)$ with
$$
R^t=e^{(Lip(\gen)+\|F_{ext}\|_{L^\infty(\bR^d)}+\eta\Lip(\chi))t}\left(R^0+\Lip(\gen)+\|F_{ext}\|_{L^\infty(\bR^d)}+\eta\Lip(\chi)\right).
$$

Finally, if $\rho_N^t,\rod_N^t$ are the solutions of $(LV)$ with initial conditions $\rhoi_N,\tau^{in}_N$ invariant by permutations in the sense that
$$
\rhoi_N\circ\sigma=\rhoi_N,\ \tau^{in}_N\circ\sigma=\tau^{in}_N,\ \forall\sigma\in\Sigma_N,
$$
then the following Dobrushin  type estimate holds true
$$
\MKd((\rho_N^t)_{N:1},(\rod_N^t)_{N:1})^2\leq 2e^{\Gamma_N(t)}\MKd((\rhoi_N)_{N:1},(\rod_N)_{N:1})^2
$$
where $\Gamma_N(t)$ is given below by \eqref{defgamN}.
\end{Thm}

\begin{proof}
The proof of Theorem \ref{thmliouvillevlasov} is easily attainable by a straightforward   modification of the one of Theorem \ref{thmvlasov}. 
 Let us define the vector $G$ with components $G_i,i=1,\dots, N$. This time, $G$ has the form
\be\label{defbvau}
G(t,Z)=\vlavn{t}{Z}{\rhoi}\in\bR^N
\ee
where we recall that $\histn{\rho^{in}}{t}: s\in[0,t]\to \rho^s$ solution of $(LV)$ with initial data $(\rho^{in})^{\otimes N}$.

One easily check that Lemme \ref{commeprt} still holds true for $\vlaun$ with the same constants $L',M',K'$.
\begin{Lem}\label{commeprtn}
For any $T\geq 0$, there exist $L'.M',K'<\infty$ such that, for any $t,t_1,t_2\leq T$, $z,z'\in\bR^{2d}$ and any $\rho^{in},\rho^{in}_1,\rho^{in}_2\in\cP(\bR^{2d})$,
\begin{eqnarray}
\|
\vlavn{t}{Z}{\rhoi}-
\vlavn{t}{Z'}{\rhoi}\|&\leq&L'\|Z-Z'\|,\ \nonumber\\
\|\vlavn{t}{Z}{\rhoi}\|&\leq & M'(1+\|Z\|)\nonumber\\
\|
\vlavn{t}{Z}{\rhoi_1}-
\vlavn{t}{Z}{\rhoi_2}
\|&\leq&
K' \sup_{
s\leq \min(t_1,t_2)}\MKu((\rho_{1N}^{s})_{N:1},(\rho_{2N}^{s})_{N:1})\nonumber\\
&&+\eta\|\nabla\chi\|_{L^\infty
}|t_1-t_2|
.\nonumber
\end{eqnarray}
where $\rho_{1N}^t,\rho_{2N}^t$ are the solutions of $(LV)$ with initial conditions $(\rhoi_1)^{\otimes N},(\rhoi_2)^{\otimes N}$.

Here $\MKu$ is the Wasserstein distance of order $1$ defined in Definition \ref{defwasun}.
\end{Lem}

For $T>0,k\in\bN$ we define $\ta_k=T2^{-k}$ and, with a slight abuse of notation, $\rho^t_k $ by $\rho^{t=0}_k=\rhoin$ and, for $l=0,\dots,2^k-1,\ u\in[0,\tau_k)$ (remember $Z:=(X,V)\in\bR^{2dN}$),

$$
(LV_{k}) \left\{
\begin{array}{l}
\partial_u\rho_k^{l\ta_k+u}(Z)+V\cdot\nabla_X\rho_k^{l\ta_k+u}(Z)=
\nabla_V\cdot(\vlavkn{t}{Z}{\rhoi}\rho_k^{l\ta_k+u}(Z)
\\ \\
\vlavkni{t}{(Y,W)}{\rhoi}
=\frac1N\sum\limits_{j=1}^N\gamma(w_i-w_j,y_i-y_j)+\eta\nabla_{z}\Psi^t(z)|_{z=y_i}
+F_{ext}(y_i),
\\ \\
\partial_s\Psi^s(z)=D\Delta_z\Psi-\kappa\Psi +g(z,(\rho^s_k)_{N;1}),\ s\in[0,t],
\\ \\
\end{array}
\right.\nonumber
$$
where $\histk{\rho^{in}}{t}: s\in[0,t]\to \rho^s_k$ solution of $(V_k)$ with initial data $\rho^{in}$.

As before, $\vlaun^k$ satisfies the same estimates than $\vlaun$ and we have the following result.

\begin{Cor}\label{commeprtnk}
For any $T\geq 0$, $k=1,\dots,N$,  $t,t_1,t_2\leq T$, $Z,Z'\in\bR^{2dN}$ and  $\rho^{in},\rho^{in}_1,\rho^{in}_2\in\cP(\bR^{2d})$,,
\begin{eqnarray}
\|
\vlavkn{t}{Z}{\rhoi}-
\vlavkn{t}{Z'}{\rhoi}\|&\leq&L'\|Z-Z'\|,\ 2\nonumber\\
\|\vlavkn{t}{Z}{\rhoi}\|&\leq & M'(1+\|Z\|)\nonumber\\
\|
\vlavkn{t}{Z}{\rhoi_1}-
\vlavkn{t}{Z}{\rhoi_2}
\|&\leq&
K' \sup_{
s\leq \min(t_1,t_2)}\MKu((\rho_{1k}^{s})_{N:1},(\rho_{2k}^{s})_{N:1})\nonumber\\
&&+\eta\|\nabla\chi\|_{L^\infty
}|t_1-t_2|
.\nonumber
\end{eqnarray}
where $\rho_{1k}^t,\rho_{2k}^t$ are the solutions of $(LV_k)$ with initial conditions $(\rhoi_1)^{\otimes N},(\rhoi_2)^{\otimes N}$.
\end{Cor}

At this point, we remark that  the proof of the existence of the solution and the estimate on the size of its support in Theorem \ref{thmvlasov} uses only the content of Corollary \ref{commeprtk}. Since Corollary \ref{commeprtnk} holds true with the same constants $L',M',K'$,  we conclude  that the proof of existence and size of the support in  Theorem \ref{thmliouvillevlasov} is exactly the same Therefore we omit it.


Uniqueness of the solution of $(LV)$ is again  a consequence of 
the Dobrushin stability result for the system $(LV)$ whose proof is  a straightforward adaptation  as the one of Theorem  \ref{thmvlasov}:

Take two initial conditions $\rhoi_{lN},l=1.2$ for $(LV)$.

By Theorem \ref{thmpropwas}, the same argument as at  the end of the proof  Theorem  \ref{thmvlasov} and  the third inequality in Lemma \ref{commeprtn} we get easily  that
\begin{eqnarray}
\MKd((\rho_N^t)_{N:1},(\rod_N^t)_{N:1})^2&\leq&
e^{\int_0^tL_N(s)ds}
\MKd((\rho_N^{in})_{N:1},(\rod_N^{in})_{N:1})^
2\nonumber\\
&+&\int_0^t
e^{\int_s^tL_N(u)du}(K')^2
\sup_{u\leq s}
\MKd((\rho_N^u)_{N:1},(\rod_N^u)_{N:1})^2
ds\
\nonumber
\end{eqnarray}
with
\begin{eqnarray}\label{lvlasovN}
L_1(t)&=&2+
2(L')^2.
\end{eqnarray}
By the Gr\"onwall Lemma again we get
$$
\MKd((\rho_N^t)_{N:1},(\rod_N^t)_{N:1})^2\leq 2e^{\Gamma_N(t)}\MKd((\rhoi_N)_{N:1},(\rod_N^{in})_{N:1})^2
$$

with
\be\label{defgamN}
\Gamma_N(t):=
t\left(
2+2(L')^2+(2K')^2e^{2t)(1+(L')^2)}\right)
\ee
and Theorem \ref{thmliouvillevlasov} is proved.
\end{proof}
\begin{Rmk}\label{rmkgammaN}
$\Gamma_N$ is independent from the two initial data $\rhoi_N$ and $\rod_N^{in}$.
\end{Rmk}
\begin{appendix}
\section{Proof of Theorem \ref{thmpropwas}}\label{proofthmpropwas}
We will denote $X=(x_1,\dots,x_N),V=(v_1,\dots,v_N),Y=(y_1,\dots,y_N), \Xi=(\xi_1,\dots,\xi_N)$, all of them belonging to $\bR^{2dN}$.

Let $\pi^{in}$  be an optimal coupling for $\rhoi_1,\rhoi_2$.
Obviously $\pi_N^{in}:=(\pi^{in})^{\otimes N}$ is a coupling for $(\rho^{in}_1)^{\otimes N},(\rho^{in}_2)^{\otimes N}$.

{Let moreover  $\rho_i^t
,\ i=1,2,$  be two   solutions  of the equations \eqref{twoeq} and let $\pi_N^{t}
$ be the unique (measure) solution to the following linear transport equation}
\begin{eqnarray}\label{lintrans}\nonumber
&&\partial_t\pi_N+V\cdot\nabla_X\pi_N+\Xi\cdot\nabla_Y\pi_N=
\nabla_V\cdot(\vlau_1([\rhoi_1]^{\leq t},X,V)\pi^t_N+\nabla_\Xi\cdot(\vlau_2([\rhoi_2]^{\leq t},Y,\Xi)\pi^t_N)
\end{eqnarray}
with $\pi_N^{t=0}=\pi_N^{op}
\in\cP_c^p(\bR^{2dN})^{\otimes 2}
$ 
optimal coupling between $\rhoi_1$ and $\rhoi_2$ invariant by permutation in the sense that, for all $\sigma\in\Sigma_N$,
$$
\pi^{in}_N(\sigma(dZ_1),\sigma(dZ_2))=\pi^{in}_N.
$$

The following first Lemma is equivalent to  \cite[Lemma 3.1]{rt}. It consists  in evolving  $\pi^{in}_N$ by the two dynamics of $\rho^t_i,\ i=1,2$. 
The proof is very similar to the one of \cite[Lemma 3.1]{rt}.

\begin{Lem}\label{lemone}

For all $t\in\bR$, $\pi_N^{t}$ is a coupling between $\rho^{t}_1$ and $\rho^t_2$.
\end{Lem}
\begin{proof}
One easily check that the two marginals of $\pi^t_N$ satisfy the two equations \eqref{twoeq}. Therefore, the Lemma holds true by unicity of these solutions. 
\end{proof}

{\bf In order to shorten a bit the notations, we will use in the sequel the following abuses of notation $\vlau_i(t,X,V):=\vlau_i([\rhoi_i]^{\leq t},X,V)$ and  $\vlau^t_i=\vlau_i(t,\cdot,\cdot)$.}
\vskip 0.5cm
By a slight modification of the proof of \cite[Lemma 3.2]{rt} we arrive easily to the following. 

\begin{Lem}\label{lemtwo}
Let 
\begin{eqnarray}
D_N(t)&:=&
\frac1N\int
 ((X-Y)^2+(V-\Xi)^2)
\pi_N^{t}(dXdVdYd\Xi)\,.\nonumber
\end{eqnarray}
Then
\begin{eqnarray}
\frac{dD_N}{dt}&\le&
{ L_1}(t) D_N
+
\frac{1}N\int_{]bR^{2dN}}
|\vlau_1(t,Y,\Xi)-\vlau_2(t,Y,\Xi)|^2
\rho_2^t(dY,d\Xi)
\nonumber
\end{eqnarray}
and
\begin{eqnarray}
\frac{dD_N}{dt}&\le&
{ L_2}(t) D_N
+
\frac{1}N\int_{]bR^{2dN}}
|\vlau_1(t,X,V)-\vlau_2(t,X,V)|^2
\rho_1^t(dX,dV),
\nonumber
\end{eqnarray}
where 
\begin{eqnarray}
L_1(t)&=&
2(1+ \sup_{(X,V)\in\ supp(\rho_1^t)}
(\Lip(\vlau_1^t
)_{(X,V)})^2),\nonumber\\
L_2(t)&=&
2(1+ \sup_{(Y,\Xi)\in\ supp(\rho_2^t)}
(\Lip(\vlau_2^t
)_{(Y,\Xi)})^2).\nonumber
\end{eqnarray}
\end{Lem}

\begin{proof}
As already mentioned, the proof is very similar to the one of \cite[Lemma 3.2]{rt}. Plugging into the definition of $D_N(t)$ the equation \eqref{lintrans} satisfied by $\pi_N^t$, integrating by part  and  using the fact that $2U\cdot V\leq U^2+V^2, \forall\ U,V\in\bR^{2d}$ (see the proof of \cite[Lemma 3.2]{rt} for details), we get
\begin{eqnarray}
\frac{dD_N}{dt}
&=&
\frac1N\int
 ((X-Y)^2+(V-\Xi)^2)
\frac{d\pi_N^{t}}{dt}(dXdVdYd\Xi)\,.\nonumber\\
&\le&\frac1N
\int_{\bR^{4dN}} \left(2(|X-Y|^2+|V-\Xi|^2){\color{white}\frac{1}N\sum_{j=1}^N}\right.\nonumber\\
&&\left.
{\color{white}\int}+
|\vlau_1(t,X,V)-\vlau_2(t,Y,\Xi)|^2
\right)\pi_N^t(dX,dV,dY,d\Xi)\nonumber\\
&\leq &
2D_N(t)+\frac2N\int_{\bR^{4dN}}
|\vlau_1(t,X,V)-\vlau_1(t,Y,\Xi)|^2
\pi_N^t(dX,dV,dY,d\Xi)\nonumber\\
&+&
\frac2N\int_{\bR^{4dN}}
|\vlau_1(t,Y,\Xi)-\vlau_2(t,Y,\Xi)|^2
\pi_N^t(dX,dV,dY,d\Xi)\nonumber\\
&\leq &
2D_N(t)\nonumber\\
&+&\frac2N
\sup_{(X,V,Y,\Xi)\in\ supp(\pi_N^t)}
(\Lip(\vlau_1^t
)_{(X,V)})^2
\int_{\bR^{4dN}}
(|X-Y|^2+|Y-\Xi|^2)
\pi_N^t(dX,dV,dY,d\Xi)\nonumber\\
&+&
\frac2N\int_{\bR^{4dN}}
|\vlau_1(t,Y,\Xi)-\vlau_2(t,Y,\Xi)|^2
\pi_N^t(dX,dV,dY,d\Xi)\nonumber\\
&\leq& 
L_1(t)D_N(t)+
\frac{2}N\int_{\bR^{2dN}}
|\vlau_1(t,Y.\Xi)-\vlau_2(t,Y,\Xi)|^2
\rho_2^t(dY,d\Xi).\nonumber
\end{eqnarray}

We conclude by the fact that 
\begin{eqnarray}
\sup\limits_{(X,V,Y,\Xi)\in supp(\pi^t_N)}(\dots)&\leq& 
\sup\limits_{(X,V)\in\ supp(\rho_1^t)}(\dots)\nonumber\\
&\leq&
\sup\limits_{(Y,\Xi)\in\ supp(\rho_2^t)}(\dots)\nonumber\
\end{eqnarray}
Indeed, (we write the argument in the case $\pi^t_N$ is a density for sake of simplicity)
\begin{eqnarray}
(X,V,Y,\Xi)\in supp(\pi^t_N)&\Rightarrow&\pi^t_N(X,V,Y,\Xi)>0\nonumber\\
&\Rightarrow&\int_{\bR^{2dN}}\pi^t_N(X,V,dY,d\Xi), \int_{\bR^{2dN}}\pi^t_N(dX,dV,Y,\Xi)>0\nonumber\\
&\Leftrightarrow& \rho_1^t(X,V),\ \rho_2^t(Y,\Xi)>0\nonumber\\
&\Leftrightarrow& (X,V)\in\ supp(\rho_1^t),\ (Y,\Xi)\in\ supp(\rho_2^t).\nonumber
\end{eqnarray}

The second inequality in Lemma \ref{lemtwo} is proved by exchanging $\rhoi_1$ and $\rhoi_2$.
\end{proof}

Therefore, by Gr\"onwall Lemma,
\begin{eqnarray}
D_N(t)&\leq& e^{\int_0^tL_1(s)ds}
D_N(0)\nonumber\\
&+&
\frac2N
\int_0^t\int_{\bR^{2dN}}|\vlau_1(s,Y,\Xi)-\vlau_2(s,Y,\Xi)| \rho_2^s(dY,d\Xi)e^{\int_s^tl_1(u)du}ds\label{oufff}
\end{eqnarray}
 
 Thanks to the following lemma, we can always suppose that $\pi^{in}_N$ is invariant by permutation in the sense of Theorem \ref{thmliouvillevlasov}.
 \begin{Lem}\label{labsig}
 Let $\rhoi_1(dZ_1),\rhoi_2(dZ_2)$ be invariant by permutation in the sense of 
 \eqref{defperm}. Then there exists an optimal coupling $\pi$ of $\rhoi_1,\rhoi_2$ invariant by permutation in the  sense that, for all $\sigma\in\Sigma_N$, $\pi\circ\sigma^{\otimes 2}=\pi$, that is
 $$
 \pi(\sigma(dZ_1,\sigma(Z_2))=\pi(dZ_1,dZ_2),\ \forall \sigma\in\Sigma_N.
 $$
 \end{Lem}
 \begin{proof}
 Let $\pi^I(dZ_1,dZ_2)$ an optimal coupling of $\rhoi_1,\rhoi_2$. It is straightforward to show that, for all $\sigma\in\Sigma_N$,
 $$
 \pi^\sigma(dZ_1,dZ_2):=\pi^I(\sigma(dZ_1),\sigma(dZ_2))
 $$ 
 is also an optimal coupling of $\rhoi_1,\rhoi_2$, since $\rhoi_1,\rhoi_2$ and the cost function is also invariant by permutation:
 $$
 |Z_1-Z_2|^2=|\sigma(Z_1)-\sigma(Z_2)|^,\ \forall \sigma\in \Sigma_N.
 $$
 Therefore 
 $$
 \pi(dZ_1,dZ_2):=\frac1{N!}\sum_{\sigma\in\Sigma_N}\pi^\sigma
 $$
 is an optimal coupling of $\rhoi_1,\rhoi_2$, and is invaraint by permutation.
 \end{proof}
We now remark that, since both $\pi^{in}_N$ and $\vlau_1,\vlau_2$ are invariant by permutations of the variables $(x_j,v_j),\ j=1,\dots, N$, so is $\pi^t_N$ for all $t\in\bR$. This implies that, in fact,
\begin{eqnarray}
D_N(t)&=&
\frac1N\int
 ((X-Y)^2+(V-\Xi)^2)
\pi_N^{t}(dXdVdYd\Xi)\,.\nonumber\\
&=&\frac1N\sum_{i=1}^N\int
 ((x_i-y_i)^2+(v_i-\xi_i)^2)
\pi_N^{t}(dXdVdYd\Xi)\,.\nonumber\\
&=&\int
 ((x_1-y_1)^2+(v_1-\xi_1)^2)
\pi_N^{t}(dXdVdYd\Xi)\,.\nonumber\\
&=&\int_{\bR^{2d}}
(|x-v|^2+|y-\xi|^2)(\pi^t_{N})_{1}(dxdv),\nonumber
\end{eqnarray}
where $(\pi^t_{N})_{1}$ is the measure on $\bR^{2d}\times \bR^{2d}$ defined, for every test function $\varphi(x,v;y,\xi)$, by
$$
\int _{\bR^{2dN}\times\bR^{2dN}}\varphi(x_1,v_1,y_1,\xi_1)\pi_N(dXdVdYd\Xi)
=
\int_{\bR^{2d}\times\bR^{2d}}\varphi(x,v,y,\xi)(\pi^t_{N})_{1}(dx,dv,dy,d\xi).
$$
Moreover,  
$(\pi^t_{N})_1$ is a coupling between $(\rho^t_1)_{N:1}$ and $(\rho^t_2)_{N:1}$. Indeed, for any test functions $\psi$ and $ \varphi$,
\begin{eqnarray}
&&\int_{\bR^{2d}\times\bR^{2d}}(\psi(x,v)+\varphi(y,\xi))(\pi^t_{N})_{1}(dx,dv,dy,d\xi)\nonumber\\
&=&
\int _{\bR^{2dN}\times\bR^{2dN}}(\psi(x_1,v_1)+\varphi(y_1,\xi_1))\pi_N(dX,dV,dY,d\Xi)\nonumber\\
&=&
\int _{\bR^{2dN}}\psi(x_1,v_1)\rho_1^t(dX,dV)
+
\int _{\bR^{2dN}}\varphi(y_1,\xi_1)\rho_2^t(dY,d\Xi)\nonumber\\
&=&
\int _{\bR^{2d}}\psi(x,v)(\rho_1^t)_{N,1}(dx,dv)
+\int _{\bR^{2d}}\varphi(y,\xi))(\rho_2^t)_{N,1}(dy,d\xi)
\nonumber
\end{eqnarray}
Consequently,  one has
\begin{eqnarray}\nonumber
\MKd((\rho^t_1)_{N:1},(\rho^t_2)_{N:1})
&:=&
\inf_{\pi\ coupling\ (\rho^t_1)_{N:1}\ and \ (\rho^t_2)_{N:1}}\int (|x-v|^2+|y-\xi|^2)\pi(dxdydvd\xi)\nonumber\\
&\leq & \int_{\bR^{2d}}
(|x-v|^2+|y-\xi|^2)(\pi^t_{N})_{1}(dxdydvd\xi)=D_N(t)\nonumber
\end{eqnarray}
and the conclusion follows by \eqref{oufff}.


\end{appendix}

\vskip 1cm
\textbf{Acknowledgements}. We warmly thank Francesco Rossi for helpful discussions. The work of TP was partly supported by LIA LYSM (co-funded by
AMU, CNRS, ECM and INdAM) and IAC (Istituto per le applicazioni del calcolo "Mauro Picone").
\vskip 1cm


\end{document}